\documentclass[12pt]{amsart}

\usepackage{adjustbox,enumerate, amsmath, amsthm, amsfonts, amssymb, xy,  mathrsfs, graphicx, paralist, fancyvrb,ytableau, tikz, tikz-cd,wrapfig, ytableau,float}
\PassOptionsToPackage{usenames, dvipsnames}{xcolor}
\usepackage[margin=1in]{geometry} 
\usepackage[bookmarks, colorlinks=true, linkcolor=blue, citecolor=blue, urlcolor=blue]{hyperref}
\usetikzlibrary{cd}

\usepackage{adjustbox}

\input xy
\xyoption{all}

\numberwithin{equation}{section}
\newtheorem{theorem}[equation]{Theorem}

\newtheorem{proposition}[equation]{Proposition}
\newtheorem{lemma}[equation]{Lemma}
\newtheorem{corollary}[equation]{Corollary}

\theoremstyle{definition}
\newtheorem{rmk}[equation]{Remark}
\newenvironment{remark}[1][]{\begin{rmk}[#1] \pushQED{\qed}}{\popQED \end{rmk}}
\newtheorem{eg}[equation]{Example}
\newenvironment{example}[1][]{\begin{eg}[#1] \pushQED{\qed}}{\popQED \end{eg}}
\newtheorem{defn}[equation]{Definition}
\newenvironment{definition}[1][]{\begin{defn}[#1]\pushQED{\qed}}{\popQED \end{defn}}
\newtheorem{notn}[equation]{Notation}
\newenvironment{notation}[1][]{\begin{notn}[#1]\pushQED{\qed}}{\popQED \end{notn}}

\newenvironment{subeqns}[1][]{\addtocounter{equation}{-1}
\begin{subequations}

}{\end{subequations}}

\newcommand{\rA}{\mathrm{A}}

\newcommand{\rB}{\mathrm{B}}

\newcommand{\bC}{\mathbf{C}}

\newcommand{\rC}{\mathrm{C}}

\newcommand{\rD}{\mathrm{D}}

\newcommand{\bH}{\mathbf{H}}

\newcommand{\rH}{\mathrm{H}}

\newcommand{\bO}{\mathbf{O}}

\newcommand{\cP}{\mathcal{P}}

\newcommand{\bQ}{\mathbf{Q}}

\newcommand{\bS}{\mathbf{S}}

\newcommand{\fS}{\mathfrak{S}}

\newcommand{\bT}{\mathbf{T}}

\newcommand{\rU}{\mathrm{U}}

\newcommand{\bV}{\mathbf{V}}

\newcommand{\rX}{\mathrm{X}}

\newcommand{\bZ}{\mathbf{Z}}

\newcommand{\fg}{\mathfrak{g}}

\newcommand{\fh}{\mathfrak{h}}

\newcommand{\bk}{\mathbf{k}}

\newcommand{\fl}{\mathfrak{l}}

\newcommand{\fn}{\mathfrak{n}}

\newcommand{\fp}{\mathfrak{p}}

\newcommand{\ft}{\mathfrak{t}}

\newcommand{\so}{\mathfrak{so}}

\newcommand{\defi}[1]{{\bf\upshape\sffamily #1}}

\renewcommand{\phi}{\varphi}
\renewcommand{\emptyset}{\varnothing}
\newcommand{\eps}{\varepsilon}

\makeatletter
\def\Ddots{\mathinner{\mkern1mu\raise\p@
\vbox{\kern7\p@\hbox{.}}\mkern2mu
\raise4\p@\hbox{.}\mkern2mu\raise7\p@\hbox{.}\mkern1mu}}
\makeatother

\DeclareMathOperator{\diag}{diag}
\DeclareMathOperator{\rank}{rank}

\DeclareMathOperator{\Sym}{Sym}

\DeclareMathOperator{\Tor}{Tor}

\newcommand{\GL}{\mathbf{GL}}

\newcommand{\Sp}{\mathbf{Sp}}

\newcommand{\Spin}{\mathbf{Spin}}
\newcommand{\Pin}{\mathbf{Pin}}

\newcommand{\ch}{\operatorname{char}}

\newcommand{\fsl}{\mathfrak{sl}}
\newcommand{\fgl}{\mathfrak{gl}}
\newcommand{\fso}{\mathfrak{so}}
\newcommand{\fsp}{\mathfrak{sp}}

\usepackage{tikz}
\usepackage{makecell}

\newcommand{\inv}{\mathrm{inv}}

\date{April 9, 2026}

\title[Kostant $\rho$-decomposition of homology I.]{Kostant $\rho$-decomposition of homology I. \\
  Finite-dimensional representations}

\author{Steven V Sam}
\address{Department of Mathematics, University of California San Diego, La Jolla, CA, USA}
\email{ssam@ucsd.edu}

\author{Keller VandeBogert}
\address{Department of Mathematics, University of Kentucky, Lexington, KY, USA}
\email{keller.v@uky.edu}

\author{Jerzy Weyman}
\address{Department of Mathematics, Jagiellonian University, Krak\'ow, Poland}
\email{jerzy.weyman@gmail.com}

\begin{document}

\maketitle

\begin{abstract}
  We give explicit, uniform formulas for the graded characters and total ranks of the Lie algebra homology of finite-dimensional representations in all classical types.   In many cases, these compute the Tor groups of finite length modules over polynomial rings, and is the first in a series of papers to investigate total rank conjectures from this perspective. These formulas refine and generalize the classical $\rho$-decomposition of Kostant, and in particular we prove that the characters involved exhibit three structural phenomena: divisibility (by a large power of 2), equidistribution, and uniform factorization formulas.
\end{abstract}

\setcounter{tocdepth}{1}
\tableofcontents

\section{Introduction}

The purpose of this series of papers is to develop a more robust set of tools for the computation of the total characters of Lie algebra homology, by which we mean the graded sum $\sum_{i\ge 0} \ch\, \rH_i(\fg; -)\, t^i$ and its specialization at $t=1$. Our motivation is twofold:
\begin{enumerate}
    \item The investigation of conjectured ``equidistribution" phenomena relating to the syzygies of coordinate rings of nilpotent orbit closures, and
    \item The construction of new counterexamples to certain rank conjectures arising in commutative algebra and algebraic topology.
\end{enumerate}
In this first paper, we will lay the groundwork in the most basic setting: computing the total characters of the Lie algebra homology in the case of finite-dimensional representations of a reductive Lie algebra.

The total ranks of homology, while not a K-theoretic invariant, are important objects of study for their relation to the topological, algebraic, and representation-theoretic ``rank conjectures". Algebraically, these conjectures (many of which are now theorems \cite{walkerTRC,vwTRC}) state that if a module has a large annihilator, then the total rank of its minimal free resolution must also be ``large" (we will refrain from a precise definition of ``large" for now). In topology, it says that if a topological space has a ``nice" action by a large torus, then the total rank of the rational homology must also be large \cite{halp1,carlsson}. In representation theory, these conjectures say that Lie algebras with a large center must also have large Lie algebra homology \cite{halp2,tir,gjranks}. The measure of ``large" in all of these contexts is always a numerical lower bound determined by some power of $2$.

Another lesser-known context where powers of $2$ arise is from a phenomenon that we have termed \defi{equidistribution}, which is where the building blocks of certain types of homology or Tor groups are highly redundant in nature (see Examples \ref{ex:typeAgrouping}, \ref{ex:typeCgrouping}). In practice, this ``redundancy" is often controlled by a power of $2$, reminiscent of the powers of $2$ alluded to above regarding the total/toral rank conjectures. An early manifestation of this equidistribution phenomenon was discovered by Kostant \cite{kostant-clifford}. The goal of this paper is to show that the powers of $2$ arising from both total rank and equidistributivity phenomena are not just surface-level similarities, but seem to be very closely related notions. Our main vessel for this are the following facts for certain distinguished parabolic decompositions in the Lie types A, B, C, and D:
\begin{itemize}
    \item We prove a refined version of equidistributivity which shows that even after taking account of homological degrees, the \emph{graded} total character of Lie algebra homology of any finite-dimensional representation is divisible by a large power of $t+1$ (where the variable $t$ is keeping track of homological degree). 
    \item We prove that there exist uniform regroupings of the representations appearing in the Lie algebra homology of an arbitrary finite-dimensional representation in such a way that each group is isomorphic to the \emph{same} fixed representation.
    \item We prove that the number of groups appearing in the above regrouping is always a power of $2$, and after combining these terms we obtain closed-form product formulas for the total character (and ranks) of the Lie algebra homology of a finite-dimensional representation that generalize the classical $\rho$-decompositions of Kostant. 
\end{itemize}

\subsection{A Generalized $\rho$-Decomposition}

Let $\fg$ be a semisimple (reductive suffices with small modifications) Lie algebra of rank $n$. We let $L^\fg_\lambda$ be the irreducible $\fg$-representation with highest weight $\lambda$. Kostant proved the following identity, which we will refer to as the $\rho$-decomposition:
\[
  \bigwedge^\bullet\fg \cong (L^\fg_\rho \otimes L^\fg_\rho)^{\oplus 2^n}.
\]
Here $\rho$ is the half-sum of all positive roots. See \cite{kostant, kostant-clifford}.
The idea is rather simple: Kostant shows (see \cite[(5.9.5)]{kostant}) that the multiset of weights of $L^\fg_\rho$ is $\{ \frac12 (\pm \beta_1 \pm \cdots \pm \beta_N) \}$ where $\{\beta_1,\dots,\beta_N\}$ are the positive roots of $\fg$ and we range over all $2^N$ sign choices. Hence the character of $L^\fg_\rho \otimes L^\fg_\rho$ agrees with $\bigwedge^\bullet(\fn \oplus \fn^*)$ where $\fn$ is the sum of the positive root spaces, and the $2^n$ comes from tensoring with $\bigwedge^\bullet \fh$ where $\fh$ is the Cartan subalgebra.

In our setting, we will take advantage of another result of Kostant describing the Lie algebra homology over nilpotent subalgebras \cite{kostant}. More precisely, we choose the following (see also Table \ref{tab:parabolics}); note here that $\eps \in \{ 0 ,1 \}$ is used to distinguish odd/even cases:

\begin{enumerate}
    \item[\textbf{Type A}] For $\fgl_{2n + \eps}$, choose the parabolic subalgebra $\fp$ with Levi factor \( \mathfrak{l} \cong \mathfrak{gl}_{n+\eps} \times \mathfrak{gl}_{n} \).
    \item[\textbf{Type B}] For $\fso_{2m+1}$, choose the parabolic subalgebra $\fp$ with Levi factor \( \mathfrak{l} \cong \fgl_m \).
    \item[\textbf{Type C}] For $\fsp_{2m}$, choose the  parabolic subalgebra $\fp$ with Levi factor \( \mathfrak{l} \cong \fgl_m \).
    \item[\textbf{Type D}] For $\fso_{2m}$, choose the  parabolic subalgebra $\fp$ with Levi factor \( \mathfrak{l} \cong \fgl_m \).
\end{enumerate}
In the type A setting, this is the parabolic decomposition induced by isolating the node that splits the Dynkin diagram as ``evenly'' as possible. In the type B/C/D cases, these are all the parabolic decompositions induced by isolating the Dynkin node corresponding to the ``last" simple root (i.e, either $\eps_m$, $2 \eps_m$, or $\eps_{m-1} + \eps_{m}$, respectively). These are the distinguished parabolics whose Levi factor is of type A, so that after applying Kostant's theorem, the resulting Levi representations are governed by symmetric function combinatorics; at the same time, their nilradicals are the classical representations \(V^*\otimes U\), \(V\oplus \bigwedge^2 V\), \(\Sym^2V\), and \(\bigwedge^2V\) that appear throughout the paper. 

Our main result shows that the total characters of the Lie algebra homology of any finite dimensional representation over the nilpotent subalgebra have closed forms that closely mirror the original decomposition proved by Kostant. For technical reasons, the precise versions of these statements need to use highest weight representations of Pin groups in the type D cases, but for the sake of readability we will ignore this fact temporarily:

\begin{theorem}[Generalized $\rho$-decomposition]\label{thm:introMainTheorem}
Set $m = 2n + \eps$ where $\eps \in \{ 0 ,1 \}$. Let $\fg$ be classical of type $\mathrm{A,B,C,D}$ with a  parabolic subalgebra $\fp=\fl\oplus\fn$ as in Table~\ref{tab:parabolics}, and let $L^\fg_\lambda$ be the finite-dimensional irreducible of highest weight $\lambda$. Then:

\begin{enumerate}
\item[\textbf{(Divisibility)}] The graded Lie homology character is divisible by a power of $1+t$:
\[
\sum_{i\ge0}\ch\,\rH_i(\fn_-;L^\fg_\lambda)\,t^i \;=\; (1+t)^{n}\cdot \Phi_\lambda(t),
\]
for some character-valued polynomial $\Phi_\lambda(t)$.

\item[\textbf{(Equidistribution)}] After specializing $t=1$, there is a canonical partition of the summation terms into $2^n$ blocks, all with the same total character. In particular
\[
\sum_{i\ge0}\ch\,\rH_i(\fn_-;L^\fg_\lambda)\;=\;2^n\cdot \Xi_\lambda,
\]
for an explicitly described character $\Xi_\lambda$.

\item[\textbf{(Product formula)}] Let $(G_{\mathrm{top}},G_{\mathrm{bot}})$ be as in Table~\ref{tab:parabolics}. There exist ``top'' and ``bottom'' weights $\rho^{\mathrm{top}}_\lambda$ and $\rho^{\mathrm{bot}}_\lambda$, explicitly determined by $\lambda$, such that
\[
\sum_{i\ge0}\ch\,\rH_i(\fn_-;L^\fg_\lambda)
\;=\;2^{n} \cdot \big(\ch\,L^{G_{\mathrm{top}}}_{\rho^{\mathrm{top}}_\lambda}\big)\,
\big(\ch\,L^{G_{\mathrm{bot}}}_{\rho^{\mathrm{bot}}_\lambda}\big).
\]
\end{enumerate}
\end{theorem}

\begin{table}[h]
\centering
\begin{tabular}{c|c|c|c}
Type & $(\fl,\fn)$  & $(G_{\mathrm{top}},G_{\mathrm{bot}})$ \\
\hline
A & $(\fgl_{n+\eps}\!\times\!\fgl_n,\;V^*\!\otimes U)$  & $(\fgl_{n+\eps},\fgl_{n})$ \\
B & $(\fgl_m,\;V\oplus\wedge^2 V)$ & $(\fsp_{2n},\Pin_{2(n+\eps)})$ \\
C & $(\fgl_m,\;\Sym^2 V)$ & $(\fsp_{2n},\Pin_{2(n+\eps)})$ \\
D & $(\fgl_m,\;\wedge^2 V)$  & $(\fsp_{2n},\Pin_{2(n+\eps)})$
\end{tabular}
\caption{Parabolics and product targets. Exact recipes for
$\rho^{\mathrm{top}}_\lambda,\rho^{\mathrm{bot}}_\lambda$ may be found in \S\S\ref{sec:typeArhodecomp},  \ref{sec:typeBCDrhodecomp}. }
\label{tab:parabolics}
\end{table}
In fact, Theorem~\ref{thm:introMainTheorem} recovers Kostant's original \(\rho\)-decomposition upon setting \(\lambda=0\): the Lie algebras \(\fgl_n\), \(\fso_{2m+1}\), \(\fsp_{2m}\), and \(\fso_{2m}\) arise here as the nilradicals \(V^*\otimes U\) (with \(n=k\)), \(V\oplus \bigwedge^2V\), \(\Sym^2V\), and \(\bigwedge^2V\), respectively.\footnote{It is worth noting that Kostant's original theorem was more than a combinatorial statement: it was formulated as an algebraic statement on the structure of the Clifford algebra \(C(\fg)\).} We also prove analogous results for the natural parabolic decomposition of $\fgl_{n+k}$ with Levi factor $\fl = \fgl_n \times \fgl_k$, though there is not as simple of a factorization formula for arbitrary $k$. 

Outside of the type B case, the nilpotent subalgebras we care about are abelian, so these character formulas are computing the total ranks of the minimal free resolutions of finite dimensional modules over a polynomial ring. Even in very simple cases these complexes are nontrivial: when $k=1$ in the type A setting, we recover the pure free resolutions constructed by Eisenbud--Fl\o ystad--Weyman \cite{EFW} (see Example \ref{ex:pureFreeRes} for details). 

As alluded to above, the phenomenon of \defi{Tor equidistribution} implies that the syzygies of these modules are obtained by copy-pasting the same fixed representation, spread out amongst various homological degrees. This property is quite subtle in general, since the way these representations are spread out may be highly nontrivial (see Example \ref{ex:typeAgrouping}). Indeed, one of the key insights of this work is the fact that equidistribution only occurs after restricting to a torus that is about half of the dimension of the ``naturally'' occurring torus action.

This phenomenon seems to be more common than one might initially expect; the ubiquity of Tor equidistribution upon restricting torus actions will be the subject of forthcoming work. The formulas of Theorem \ref{thm:introMainTheorem} specialize further to explicit dimension formulas which allow us to prove that the total ranks of Lie algebra homology seem to be ``discretely'' distributed:

\begin{corollary}\label{cor:2n+2n-1 conj}
    In all of the cases of Theorem \ref{thm:introMainTheorem} where $\fn$ is abelian, there is an equality
    $$\dim \rH_\bullet (\fn_- ; L^{\fg}_\lambda) = 2^{\dim \fn} \cdot C_\lambda,$$
    where $C_\lambda \geq 1$ has the property that $C_\lambda \geq 1.5$ if $C_\lambda \neq 1$. In particular, all of the finite dimensional modules considered in this paper satisfy the ``$2^n + 2^{n-1}$ conjecture'' (see for instance \cite{charalambous,evangrif,CEM}). 
\end{corollary}

Corollary \ref{cor:2n+2n-1 conj} may be seen as evidence that the ``$2^n + 2^{n-1}$ conjecture'' holds based on the heuristic that resolutions obtained by Kostant's theorem are as small as possible for a given set of parameters (as previously mentioned, the pure free resolutions of \cite{EFW} as well as the Iyengar--Walker counterexamples \cite{IW} are both special cases of Kostant's theorem). Although the Iyengar--Walker counterexamples do not come from the parabolics considered in this paper, they are still governed by Kostant's theorem: they arise from the Heisenberg parabolic in \(\fsp_{2n+2}\), whose Levi factor is \(\fgl_1\times \fsp_{2n}\) and whose nilradical is the Heisenberg Lie algebra \(\mathfrak h=(E\oplus E^*)\oplus \bk \cdot t\). In particular, the Iyengar--Walker complex may be viewed as the specialization at \(t=0\) of the \(\rU(\mathfrak h)\)-resolution of the field \(\bk\), so it belongs to the same Lie-theoretic framework as the resolutions studied here. The total ranks for these more general choices of parabolics will be studied much more closely in future work.

In fact, Theorem \ref{thm:introMainTheorem} is a consequence of much more general combinatorial properties of certain classes of determinants which we develop in the first part of the paper, and in the second part of the paper we see how these combinatorial identities imply the homological behavior of Theorem \ref{thm:introMainTheorem}. 

\subsection{Organization} 

This paper is structured in two parts: in Part~\ref{part:1}, we perform a strictly combinatorial study of polynomials that specialize to the graded characters of Lie algebra homology that we are interested in. Remarkably, the combinatorics involved in understanding the type B/C/D cases can all be handled uniformly, while the type A case is largely disjoint from these cases. In Part~\ref{part:2} of the paper, we reinterpret the combinatorial identities in part~\ref{part:1} algebraically and deduce Theorem \ref{thm:introMainTheorem} in all of the relevant cases. We conclude with many examples illustrating the generalized $\rho$-decompositions. 

\subsection{Acknowledgements}

SS was partially supported by NSF DMS-2302149. KV was partially supported by NSF DMS-2202871. Part of this work was conducted while KV was visiting the Simons Institute for the Theory of Computing. JW was supported by Polish NCN grants MAESTRO UNO 2019/34/A/ST1/263 and OPUS DEC 2024/55/B/ST1/01437.

\part{Combinatorial Aspects} \label{part:1}

\section{Preliminaries}

We fix notation and record the character formulas and determinant identities we need.

\subsection{Fixed notation}\label{subsec:fixedNotation}

We adopt the following set-theoretic conventions for the remainder of the paper. The following notation will be used ubiquitously and often tacitly throughout this paper.

\begin{notation}\label{not:basicNotation}
For a positive integer $m$, write $[m]:=\{1,2,\dots,m\}$. The cardinality of a finite set $S$ is denoted $|S|$.  For $0\le r\le m$, let
\[
\binom{[m]}{r}=\{\,S\subseteq [m]\mid |S|=r\,\}.
\]
If $S\subseteq [m]$, its complement in $[m]$ is $S^c=[m]\setminus S$ and we define $S^{\inv}:=\{\,m+1-s\mid s\in S\,\}$. We write $\cP([m])$ for the Boolean lattice of all subsets of $[m]$.
For $\eps\in\{0,1\}$ set
\[
[m]_\eps:=\{\,i\in[m]\mid i\equiv \eps \ (\mathrm{mod}\ 2)\,\},
\qquad
S_\eps:=S\cap [m]_\eps.
\]
 We also use $[m]_{\rm odd}=[m]_1$ and $[m]_{\rm even}=[m]_0$. For $S=\{s_1<\cdots<s_r\}\subseteq [m]$, define
\[
\Sigma(S):=\sum_{i=1}^r s_i.
\]
Note that $(S^{\inv})^{\inv}=S$ and $\Sigma(S^{\inv})=|S|(m+1)-\Sigma(S)$.

For $\alpha=(\alpha_1,\dots,\alpha_m)$ and $S=\{s_1<\cdots<s_r\}\subseteq [m]$, set
\[
\alpha|_S:=(\alpha_{s_1},\dots,\alpha_{s_r}). \qedhere
\]
\end{notation}

\begin{example}
  For $m=7$ and $S=\{2,3,6\}$, we have $S_0=\{2,6\}$, $S_1=\{3\}$, and $S^{\inv}=\{6,5,2\}$ with $\Sigma(S^{\inv})=3\cdot 8-\Sigma(S)=24-11=13$.
\end{example}

\subsection{Characters} \label{sec:char}

We will work in the ring $A=\bQ[x_1^{\pm 1/2},\dots,x_m^{\pm 1/2}]$ of Laurent polynomials in $x_1^{1/2},\dots,x_m^{1/2}$. Given $\alpha \in \bZ^m \cup (\frac12 + \bZ)^m$, define:
\begin{subequations}
  \allowdisplaybreaks
  \begin{align}
  a_\alpha(x) &= \det(x_i^{\alpha_j}), &\rho^\rA &= (m-1,\dots,1,0), &s_\alpha^\rA(x) &= \frac{a_{\alpha+\rho^\rA}(x) }{a_{\rho^\rA}(x)}.\\
  b_\alpha(x) &= \det(x_i^{\alpha_j}-x_i^{-\alpha_j}), &\rho^\rB &= (m-\frac12,\dots,\frac32,\frac12), &s_\alpha^\rB(x) &= \frac{b_{\alpha+\rho^\rB}(x) }{b_{\rho^\rB}(x)}.\\
  c_\alpha(x) &= \det(x_i^{\alpha_j}-x_i^{-\alpha_j}), &\rho^\rC &= (m,\dots,2,1), &s_\alpha^\rC(x) &= \frac{c_{\alpha+\rho^\rC}(x) }{c_{\rho^\rC}(x)}.\\
  d_\alpha(x) &= \frac12 \det(x_i^{\alpha_j} + x_i^{-\alpha_j}), &\rho^\rD &= (m-1,\dots,1,0), &s_\alpha^\rD(x) &= \frac{2d_{\alpha+\rho^\rD}(x) }{d_{\rho^\rD}(x)}.  
\end{align}
\end{subequations}
For short, we denote $s^\rA_\alpha(x)$ as simply $s_\alpha(x)$. When $\alpha = \lambda$ is an integer partition, $s_\lambda(x)$ is the Schur polynomial corresponding to the partition $\lambda$. Also, $b_\alpha$ and $c_\alpha$ are evidently the same polynomials, but we use different notation to distinguish which types we are working in. 

Each of the $s_\alpha^\rX$ are elements of $A$ for $\rX \in \{\rA,\rB,\rC,\rD\}$. We will allow $m$ to vary; rather than build this into the notation, the value of $m$ is implied by the number of input variables when needed. In the following, the notation \(\ch\) denotes the formal character.

\begin{proposition}[Weyl character formulas: determinantal forms]\label{prop:WeylCharForms}
  \ 
\begin{enumerate}
\item Pick $\lambda = (\lambda_1 ,\lambda_2 , \dots , \lambda_m) \in \bZ^m$ with $\lambda_1 \ge \lambda_2 \ge \cdots \ge  \lambda_m$. Let \(L_\lambda^{\fgl_m}\) be the irreducible $\fgl_m(\bC)$-representation with highest weight $\lambda$. Then
  \[
    \ch L_\lambda^{\fgl_m} = s^\rA_\lambda (x).
  \]

\item Pick $\lambda = (\lambda_1, \dots , \lambda_m) \in \bZ^m \cup (\tfrac12 + \bZ)^m$ with $\lambda_1 \ge \cdots \ge \lambda_m \ge 0$. Let \(L_\lambda^{\so_{2m+1}}\) be the irreducible $\so_{2m+1}(\bC)$-representation with highest weight $\lambda$. Then
    \[
    \ch L_\lambda^{\so_{2m+1}} = s_\lambda^\rB (x).
    \]
    
\item Pick $\lambda =(\lambda_1, \dots , \lambda_m) \in \bZ^m$ with $\lambda_1 \ge \cdots \ge \lambda_m \ge 0$. Let \(L_\lambda^{\fsp_{2m}}\) be the irreducible $\fsp_{2m}(\bC)$-representation with highest weight $\lambda$. Then
    \[
    \ch L_\lambda^{\fsp_{2m}} = s^\rC_\lambda (x).
    \]

\item Pick $\lambda = (\lambda_1 ,\lambda_2 , \dots , \lambda_m)\in \bZ^m \cup (\tfrac12 + \bZ)^m$ with $\lambda_1 \ge \cdots \ge \lambda_m \ge 0$. Let \(L^{\Pin_{2m}}_\lambda\) be the irreducible $\Pin_{2m}(\bC)$-representation with highest weight $\lambda$. Then
    \[
        \ch L^{\Pin_{2m}}_\lambda =
        \begin{cases}
            2\,s^\rD_\lambda (x), & \text{if } \lambda_m = 0, \\
             s^\rD_\lambda (x), & \text{if } \lambda_m > 0.
        \end{cases}
    \]
\end{enumerate}
\end{proposition}

\begin{remark}
  For our purposes, the characters of $\fso_{2m}(\bC)$ are not the right thing to consider. Instead, it is more convenient to consider the Pin group $\Pin_{2m}(\bC)$, which is a double covering of the orthogonal group $\bO_{2m}(\bC)$ \cite[\S 20]{fultonharris}. Given $\lambda = (\lambda_1,\dots,\lambda_m)$ such that $\lambda_1 \ge \cdots \ge \lambda_m \ge 0$ and the $\lambda_i$ are either all integers or half-integers, we define
\[
  L_\lambda^{\Pin_{2m}} = \begin{cases} L_\lambda^{\fso_{2m}} & \text{if $\lambda_m=0$}\\
    L_{\lambda}^{\fso_{2m}} \oplus L_{(\lambda_1,\dots, \lambda_{m-1}, -\lambda_m)}^{\fso_{2m}} & \text{if $\lambda_m > 0$} \end{cases}.
\]
These turn out to be irreducible representations for $\Pin_{2m}(\bC)$ (we will not use this fact, we just state this to motivate the notation).
\end{remark}

\subsection{Identities}

\begin{lemma} \label{lem:prod}
  Each of the denominators in the character formulas from the previous section have an explicit product formula:
  \allowdisplaybreaks
  \begin{align*}
    a_{\rho^\rA} (x_1 , \dots , x_m) &= \prod_{1 \leq i < j \leq m} (x_i - x_j),\\
    b_{\rho^\rB} (x_1, \dots , x_n) & = \frac{1}{(x_1 \cdots x_n)^{n - \frac{1}{2}}} \prod_{1 \le i < j \le n}(x_i - x_j)(x_i x_j -1) \cdot \prod_{i=1}^n(x_i - 1),\\
    c_{\rho^\rC} (x_1 , \dots , x_n) &= \frac{1}{(x_1 \cdots x_n)^n} \prod_{1 \le i < j \le n}(x_i - x_j)( x_i x_j - 1) \cdot \prod_{i=1}^n(x_i^2 - 1),\\
    d_{\rho^\rD}(x_1 , \dots , x_n) &= \frac{1}{(x_1 \cdots x_n)^{n - 1}} \prod_{1 \le i < j \le n}(x_i - x_j)(x_i x_j-1).
  \end{align*}
\end{lemma}

The proofs are straightforward: the terms on the left side are Laurent polynomials which are evidently divisible by the linear factors in the right side, so it remains to show that the degrees and leading terms match up.

The next identity factors a type \(\rA\) Vandermonde into type \(\rC\) and \(\rD\) denominators; we will use it repeatedly.

\begin{lemma}\label{lem:adeltaIdentities}
Let $y_1,\dots,y_n$ be independent variables.
\begin{enumerate}
\item If $m=2n$,
\[
a_{\rho^\rA}(y_1,\dots,y_n,y_n^{-1},\dots,y_1^{-1})
= c_{\rho^\rC}(y_1,\dots,y_n)\cdot d_{\rho^\rD}(y_1,\dots,y_n).
\]
\item If $m=2n+1$,
\begin{align*}
a_{\rho^\rA}(y_1,\dots,y_n,1,y_n^{-1},\dots,y_1^{-1})
&= c_{\rho^\rC}(y_1,\dots,y_n)\cdot d_{\rho^\rD}(y_1,\dots,y_n,1)\\
&= b_{\rho^\rB}(y_1,\dots,y_n)^2 \prod_{i=1}^n\big(y_i-y_i^{-1}\big).
\end{align*}
\end{enumerate}
\end{lemma}

\begin{proof}

{\bf Proof of (1):} Specialize the Lemma~\ref{lem:prod} and reindex to get
    \begingroup\allowdisplaybreaks
    \begin{align*}
      a_{\rho^\rA} (y_1 , \dots , y_n, y_n^{-1}, \dots , y_1^{-1}) &= (-1)^{\binom{n}{2}} a_{\rho^\rA} (y_1 , \dots , y_n, y_1^{-1}, \dots , y_n^{-1})\\
      &= (-1)^{\binom{n}{2}} \prod_{1 \leq i < j \leq n} (y_i - y_j)(y_i^{-1} - y_j^{-1}) \prod_{\substack{1 \leq i \leq n, \\ 1 \leq j \leq n}} (y_i - y_j^{-1})  \\
        &= \prod_{1 \leq i < j \leq n} \frac{(y_i-y_j)^2}{y_i y_j} \prod_{1 \leq i < j \leq n} \frac{(1-y_iy_j)^2}{y_iy_j} \cdot \prod_{i=1}^n \frac{(y_i^2-1)}{y_i} \\
        &= \frac{1}{(y_1 \cdots y_n)^{2n-1}} \prod_{1 \leq i < j \leq n} (y_i - y_j)^2(1-y_iy_j)^2 \cdot \prod_{i=1}^n (y_i^2-1).
    \end{align*}
    \endgroup

{\bf Proof of (2):} We have that $a_{\rho^\rA}(y_1,\dots,y_n,1,y_n^{-1},\dots,y_1^{-1})$ equals
\begingroup\allowdisplaybreaks
\begin{align*}
  &\quad\ (-1)^n a_{\rho^\rA}(y_1,\dots,y_n,y_n^{-1},\dots,y_1^{-1},1)\\
  &= (-1)^n a_{\rho^\rA}(y_1,\dots,y_n,y_n^{-1},\dots,y_1^{-1})
     \prod_{i=1}^n (y_i-1)(y_i^{-1}-1)\\
  &= \frac{a_{\rho^\rA}(y_1,\dots,y_n,y_n^{-1},\dots,y_1^{-1})}{y_1\cdots y_n}
     \prod_{i=1}^n (y_i-1)^2\\
  &= \frac{1}{y_1\cdots y_n}
     \prod_{1\le i<j\le n}(y_i-y_j)
     \prod_{1\le i<j\le n}(y_j^{-1}-y_i^{-1})
     \prod_{i,j=1}^n (y_i-y_j^{-1})
     \prod_{i=1}^n (y_i-1)^2\\
  &= \frac{1}{(y_1\cdots y_n)^{2n}}
     \prod_{1\le i<j\le n}(y_i-y_j)^2(y_i y_j-1)^2
     \prod_{i=1}^n (y_i^2-1)(y_i-1)^2.
\end{align*}
\endgroup
Again, we finish using Lemma~\ref{lem:prod}.
\end{proof}

\begin{proposition} \phantomsection \label{prop:stair-sosp}
 The characters of the irreducible $\fsp_{2n}$-representation and $\Pin_{2n}$-representation of highest weight $\rho^\rC= (n,n-1, \dots , 1)$ coincide:
    \[
      \ch L^{\fsp_{2n}}_{\rho^\rC} = \ch L^{\Pin_{2n}}_{\rho^\rC}.
    \]
\end{proposition}

  \begin{proof}
    From Proposition~\ref{prop:WeylCharForms}, we have
\begin{align*}
  \ch L^{\fsp_{2n}}_{\rho^\rC} &=
  \frac{
    \det\left( x_j^{2n - 2i + 2} - x_j^{-(2n - 2i + 2)} \right)_{1 \le i,j \le n}
  }{
    \det\left( x_j^{n - i + 1} - x_j^{-(n - i + 1)} \right)_{1 \le i,j \le n}
  } 
  =
  \frac{
    \det\left( \sum_{k=0}^{2n-2i+1} x_j^{2n-2i+1-2k} \right)_{1 \le i,j \le n}
  }{
    \det\left( \sum_{k=0}^{n-i} x_j^{n-i-2k} \right)_{1 \le i,j \le n}
  },
\end{align*}
where in the second equality, we factored out \(x_j-x_j^{-1}\) from the \(j\)th column of both matrices.

Now we perform row operations. In the top determinant, for \(i=1,\dots,n-1\) in increasing order, subtract row \(i+1\) from row \(i\). This removes the inner terms from each sum, so the top determinant becomes
\[
\det\left( x_j^{2n-2i+1}+x_j^{-(2n-2i+1)} \right)_{1\le i,j\le n}.
\]

In the bottom determinant, for \(i=1,\dots,n-2\) in increasing order, subtract row \(i+2\) from row \(i\). This again removes the inner terms from each sum. The \((n-1)\)st row is already \(x_j+x_j^{-1}\), and the \(n\)th row is already \(1\). Thus the bottom determinant becomes
\[
\det\!\begin{pmatrix}
x_j^{n-1}+x_j^{-(n-1)} & 
x_j^{n-2}+x_j^{-(n-2)} &
\cdots &
x_j+x_j^{-1} &
1
\end{pmatrix}_{1\le j\le n}^T
=
\frac12
\det\left( x_j^{n-i}+x_j^{-(n-i)} \right)_{1\le i,j\le n},
\]
where the factor \(\frac12\) comes from the fact that in the last row we have \(1\) instead of \(2=x_j^0+x_j^{-0}\).

Therefore, using Proposition~\ref{prop:WeylCharForms}, we get
\begin{align*}
  \ch L^{\fsp_{2n}}_{\rho^\rC}
  &=
  \frac{
    \det\left( x_j^{2n-2i+1}+x_j^{-(2n-2i+1)} \right)_{1\le i,j\le n}
  }{
    \frac12\det\left( x_j^{n-i}+x_j^{-(n-i)} \right)_{1\le i,j\le n}
  } 
  = \ch L^{\Pin_{2n}}_{\rho^\rC}.  \qedhere
\end{align*}
  \end{proof}

\section{Type A Combinatorial Identities}\label{sec:typeAidentities}

\subsection{Setup} \label{sec:A-setup}
Fix a tuple $\lambda \in \bZ^{n+k}$. Throughout this section, we will use the notation
\[
  \delta_m = (m-1, m-2 , \dots , 0 ) \in \bZ^m.
\]
When the subscript $m$ is clear we will often omit it from the notation. For $S \subseteq [n+k]$ with $|S|=n$,  define
\begin{align*}
  \iota^1_\lambda(S) &= (\lambda + \delta_{n+k})|_S & \beta^1_\lambda(S) &= \iota^1_\lambda(S) - \delta_n - (k,\dots,k)\\
  \iota^2_\lambda(S) &= (\lambda + \delta_{n+k})|_{S^c} &  \beta^2_\lambda(S) &= \iota^2_\lambda(S) - \delta_k.
\end{align*}

\begin{example}
  Take \(n=2\), \(k=1\), and \(\lambda=(3,2,1)\). Then
  \[
    \delta_3=(2,1,0),
    \qquad
    \lambda+\delta_3=(5,3,1).
  \]
  For the three subsets \(S\in \binom{[3]}{2}\), we obtain:
  \[
  \begin{array}{c|c|c|c|c|c}
    S & \iota^1_\lambda(S) & \beta^1_\lambda(S) & \iota^2_\lambda(S) & \beta^2_\lambda(S) & \rank(S) \\ \hline
    \{1,2\} & (5,3) & (3,2) & (1) & (1) & 0 \\
    \{1,3\} & (5,1) & (3,0) & (3) & (3) & 1 \\
    \{2,3\} & (3,1) & (1,0) & (5) & (5) & 2
  \end{array}
  \]
  since \(\delta_2=(1,0)\), \((k,\dots,k)=(1,1)\), and \(\delta_1=(0)\).
\end{example}

\begin{example}
  More generally, if \(S=[n]\subseteq [n+k]\), then \(S^c=\{n+1,\dots,n+k\}\), and
  \begin{align*}
    \iota^1_\lambda([n])
    &=
      (\lambda_1+n+k-1,\lambda_2+n+k-2,\dots,\lambda_n+k),
    &
          \beta^1_\lambda([n])
    &=
    (\lambda_1,\dots,\lambda_n),\\
    \iota^2_\lambda([n])
    &=
    (\lambda_{n+1}+k-1,\lambda_{n+2}+k-2,\dots,\lambda_{n+k}),
    &
      \beta^2_\lambda([n])
    &=
    (\lambda_{n+1},\dots,\lambda_{n+k}).
  \end{align*}
  Finally,
    $\rank([n])=\Sigma([n])-\binom{n+1}{2}=0$.
\end{example}

It is useful to record these tuples more explicitly. If
\[
S=\{s_1<\cdots<s_n\}\subseteq [n+k]
\qquad\text{and}\qquad
S^c=\{t_1<\cdots<t_k\},
\]
then by definition
\begin{align*}
\beta^1_\lambda(S)
&=
(\lambda_{s_1}-(s_1-1),\lambda_{s_2}-(s_2-2),\dots,\lambda_{s_n}-(s_n-n)),\\
\beta^2_\lambda(S)
&=
(\lambda_{t_1}+n-(t_1-1),\lambda_{t_2}+n-(t_2-2),\dots,\lambda_{t_k}+n-(t_k-k)).
\end{align*}
Now, let $t$ be a formal variable and define
\begin{align*}
    \bH^k_\lambda (x_1 , \dots , x_n ; t) &:= \sum_{S \in \binom{[n+k]}{n}} s_{\beta^1_\lambda(S)} (x_1, \dots , x_n) s_{\beta^2_\lambda(S)} (x_1 , \dots , x_k) \cdot t^{\rank(S)}\\
    &= (x_1\cdots x_n)^{-k} \sum_{S \in \binom{[n+k]}{n}} \frac{a_{\iota^1_\lambda(S)} (x_1, \dots , x_n)}{a_{\delta_n}(x_1,\dots,x_n)}
    \frac{a_{\iota^2_\lambda(S)} (x_1 , \dots , x_k)}{a_{\delta_k}(x_1,\dots,x_k)} \cdot t^{\rank(S)},
\end{align*}
  where $\rank(S) = \Sigma(S) - \binom{|S|+1}{2}= \sum_{s \in S} s - \binom{|S|+1}{2}$. 

Similarly, given any $T \subseteq [n + k]_{\rm odd}$ we define the following variant:
\begin{align*}
    \bH^k_{\lambda,T} (x_1 , \dots , x_n) &:= \sum_{\substack{S \in \binom{[n+k]}{n} \\ S_{\rm odd} = T}}
    s_{\beta^1_\lambda(S)} (x_1, \dots , x_n) s_{\beta^2_\lambda(S)} (x_1 , \dots , x_k)\\
    &= (x_1\cdots x_n)^{-k} \sum_{\substack{S \in \binom{[n+k]}{n} \\ S_{\rm odd} = T}}\frac{a_{\iota^1_\lambda(S)} (x_1, \dots , x_n)}{a_{\delta_n}(x_1,\dots,x_n)}
    \frac{a_{\iota^2_\lambda(S)} (x_1 , \dots , x_k)}{a_{\delta_k}(x_1,\dots,x_k)}.
\end{align*}

The purpose of this section is to prove that the polynomials $\bH^k_\lambda (x; t)$ and $\bH^k_{\lambda,T} (x)$ exhibit three remarkable properties, whose proofs are entirely combinatorial in nature:
\begin{enumerate}
    \item[\textbf{Determinantal Form:}] The polynomial $\bH^k_\lambda (x_1 , \dots , x_n ; t)$ may be written as the determinant of a single matrix, and is divisible by $(t+1)^k$. 
    \item[\textbf{Equidistribution:}] The terms appearing in the sum $\bH^n_\lambda (x_1 , \dots , x_n ; t)$ are ``equidistributed'' in a very strong sense: for {\it any} two subsets $T, T' \subseteq [2n]_{\rm odd}$, we have
      \[
        \bH^n_{\lambda,T}(x_1,\dots,x_n) = \bH^n_{\lambda,T'}(x_1,\dots,x_n).
      \]
      This fails when $k \ne n$, but we prove a partial result. As a consequence, for general $k$ we conclude that
      \[
        \bH^k_{\lambda}(x_1,\dots,x_n;1) = 2^k \bH^k_{\lambda, [n+k]_{\rm odd}}(x_1,\dots,x_n).
      \]
    \item[\textbf{Factorizability:}] When $k \in \{ n-1, n\}$, there exist $\mu$ and $\nu$ determined by $\lambda$ such that
      \[
        \bH^k_{\lambda, [n+k]_{\rm odd}} (x_1 , \dots , x_n ) = s_\mu (x_1 ,\dots , x_n) s_\nu (x_1 , \dots , x_k).
      \]
\end{enumerate}

\subsection{Determinantal Form} \label{sec:det-A}

As usual, we assume $n \ge k$. Define $(n+k)\times(n+k)$ matrices $A_\lambda(x;t)$ and $A'_\lambda(x;t)$ with entries
\begin{align*}
A_\lambda(x;t)_{i,j} &= 
\begin{cases}
(-t)^{j-1}\,x_i^{\lambda_j + n + k - j} & \text{if } 1\le i \le n,\\
x_{i-n}^{\lambda_j + n + k - j} & \text{if } n<i\le n+k,
\end{cases}\\
A'_\lambda(x;t)_{i,j} &= 
\begin{cases}
(-t)^{j-1}\,x_i^{\lambda_j + n + k - j} & \text{if } 1\le i \le n,\\
\Big(\sum_{r=0}^{j-2}(-t)^r\Big)\,x_{i-n}^{\lambda_j + n + k - j} & \text{if } n<i\le n+k,
\end{cases}
\end{align*}
with the convention that $\sum_{r=0}^{-1}(\cdot)=0$ (so the $j=1$ entry in the bottom block of $A'_\lambda$ is $0$). We record the determinant form and its \((t+1)^k\) factor.

\begin{proposition} \label{prop:typeAprop}
We have
    \begin{align*}
        (-t)^{\binom{n}{2}} (x_1\cdots x_n)^k a_{\delta_n}(x_1,\dots,x_n) a_{\delta_k}(x_1,\dots,x_k) \bH^k_\lambda (x_1 , \dots , x_n ; t) &= \det A_{\lambda} (x;t) \\
        &= (t+1)^k \cdot \det A_\lambda' (x;t).
    \end{align*}
  \end{proposition}

\begin{proof}
  Let $M_S^T$ denote the submatrix of $A_\lambda(x;t)$ with rows indexed by $T$ and columns indexed by $S$. The generalized Laplace expansion along the first $n$ rows of the matrix $A_\lambda (x;t)$ gives
    \begin{align*}
      \det A_\lambda (x;t) &= \sum_{S \in \binom{[n+k]}{n}} (-1)^{\rank(S)} \det M_S^{[n]} \cdot  \det M_{S^c}^{\{n+1,\dots,n+k\}}\\
                           &= \sum_{S \in \binom{[n+k]}{n}} (-1)^{\rank(S)} (-t)^{\Sigma(S)-n} a_{\iota^1_\lambda(S)} (x_1,\dots , x_n) \cdot  a_{\iota^2_\lambda(S)} (x_1, \dots , x_k) \\
        &= (-t)^{\binom{n}{2}} \sum_{S \in \binom{[n+k]}{n}} t^{\rank(S)} a_{\iota^1_\lambda(S)} (x_1,\dots , x_n) \cdot  a_{\iota^2_\lambda(S)} (x_1, \dots , x_k) \\       
        &= (-t)^{\binom{n}{2}} (x_1\cdots x_n)^k a_{\delta_n}(x_1,\dots,x_n) a_{\delta_k}(x_1,\dots,x_k) \bH^k_\lambda (x_1 , \dots , x_n ; t).
    \end{align*}
To prove the second equality, subtract row $i$ from row $n+i$ for each $i =1 ,\dots , k$ and divide each of rows $n+1,\dots,n+k$ by $1+t$.
\end{proof}

\begin{example}
    When $n=k=2$ and $\lambda = (0)$ the matrix $A_\lambda (x;t)$ is given by
    \[
      \begin{pmatrix}
        x_{1}^{3}&-x_{1}^{2}t&x_{1}t^{2}&-t^{3}\\
        x_{2}^{3}&-x_{2}^{2}t&x_{2}t^{2}&-t^{3}\\
        x_{1}^{3}&x_{1}^{2}&x_{1}&1\\
        x_{2}^{3}&x_{2}^{2}&x_{2}&1
      \end{pmatrix}.
          \]
        Subtracting the first two rows from the last two rows (in order) gives the matrix
        \[
\begin{pmatrix}
x_1^3 & -x_1^2 t & x_1 t^2 & -t^3 \\
x_2^3 & -x_2^2 t & x_2 t^2 & -t^3 \\
0 & x_1^2 (1 + t) & x_1 (1 - t^2) & 1 + t^3 \\
0 & x_2^2 (1 + t) & x_2 (1 - t^2) & 1 + t^3
\end{pmatrix}
\]
Finally, factor out $1+t$ from each of the last two rows to obtain $A'_\lambda(x;t)$.
\end{example}

\subsection{Equidistribution}

We now consider specializing the polynomial $\bH^k_{\lambda} (x_1 , \dots , x_n; t)$ at $t=1$. At \(t=1\), grouping terms by the odd entries of \(T\) yields linear relations among the \(\bH^k_{\lambda,T}\); for \(k=n\) this forces equality for all \(T\).

\begin{proposition} \label{prop:equi-A}
    Assume $n \geq k$ and set $q := |[n+k]_{\rm odd}| = \lceil (n+k)/2 \rceil$. For any $0 \le i \le q$, we have
    \[
      \sum_{|T| = i} \bH^k_{\lambda,T} (x_1 , \dots , x_n) = \binom{k}{q-i} \bH^k_{\lambda, [n+k]_{\rm odd}} (x_1 , \dots, x_n ).
    \]
    If $n=k$, then for any subsets $T,T' \subseteq [2n]_{\rm odd}$, we have
    \[
      \bH_{\lambda, T} (x_1 , \dots , x_n) = \bH_{\lambda, T'} (x_1 , \dots , x_n).
    \]
    In all cases, we have $\bH^k_\lambda (x_1 , \dots , x_n;1 ) = 2^k \bH_{\lambda , [n+k]_{\rm odd}}^k (x_1 , \dots , x_n)$.
  \end{proposition}
  
\begin{proof}
    Pick \( \bT \subseteq [n+k]_{\rm odd} \). Define the \( (n+k) \)–by–\( (n+k) \) matrix \( M(\bT) \) with entries:
\[
M(\bT)_{i,j} = 
\begin{cases}
x_i^{\lambda_j + n + k - j} & \text{for } 1 \le i \le n, \\
x_{i - n}^{\lambda_j + n + k - j} & \text{for } n+1 \le i \le n+k \text{ and } j \notin \bT, \\
0 & \text{for } n+1 \le i \le n+k \text{ and } j \in \bT.
\end{cases}
\]
  
 By the generalized Laplace expansion (expanding along the first $n$ rows), we have 
    \begin{align*}
      \det M(\bT) &= (-1)^{\binom{n+1}{2}} \sum_{\substack{S \in \binom{[n+k]}{n}\\ S_{\rm odd} \supseteq \bT}} (-1)^{|S_{\rm odd}|} a_{\iota^1_\lambda(S)}(x_1,\dots,x_n) \cdot a_{\iota^2_\lambda(S)}(x_1,\dots,x_k).
    \end{align*}    
    Next, we claim that $\det M(\bT)=0$ if $|\bT| < k$. To see this, subtract row $n+i$ from row $i$ for each $i=1,\dots,k$. Then in the columns indexed by $[n+k] \setminus \bT$, the first $k$ entries are 0. In particular, the submatrix consisting of the first $k$ rows has rank $<k$, which proves the claim.

    In particular, if $|\bT|<k$, we have
    \begin{subeqns}
      \begin{align} \label{subeqn:binom-id}
      \sum_{\substack{T \supseteq \bT}} (-1)^{|T|} \bH^k_{\lambda,T} (x_1, \dots , x_n ) = 0.
      \end{align}
    \end{subeqns}
Let $v_{q-k},\dots,v_q$ be a set of independent variables and consider the system of linear equations
    \[
            \sum_{i = \alpha}^q (-1)^i \binom{i}{\alpha} v_i = 0, \qquad \alpha = 0,\dots,k-1.
          \]
These equations are linearly independent: in matrix form they give a generalized Vandermonde matrix
    \[
      \left( (-1)^j \binom{j}{i} \right)_{\substack{j=q-k,\dots,q-1\\ i=0,\dots,k-1}},
    \]
 which has nonzero determinant. Hence its solutions are unique up to scalar multiple; we will find two solutions with the same value for $v_q$ and use this to conclude equality.
 
First, define $u_i = \sum_{|T|=i} \bH^k_{\lambda,T} (x_1 , \dots , x_n)$. Note that $u_i=0$ if $i<q-k$ since no subset of size $n$ of $[n+k]$ can have less than $q-k$ odd members. For a given $0 \le \alpha \le k-1$, we can sum \eqref{subeqn:binom-id} for all $\bT$ such that $|\bT|=\alpha$  to conclude that $v_i=u_i$ is a solution to our system of linear equations.

Second, we claim that setting $v_i = \binom{k}{q-i} u_q$ gives another solution to the above equations. Let $z$ be a formal variable. We have
    \begin{align*}
      (-1)^\alpha (1+z)^{k-\alpha-1} &= (-1)^\alpha \frac{(1+z)^k}{(1+z)^{\alpha+1}} = (\sum_{j=0}^k \binom{k}{j} z^j ) ( \sum_{i \ge \alpha} \binom{i}{\alpha} (-1)^{i} z^{i-\alpha}).
    \end{align*}
    The coefficient of $z^{q-\alpha}$ on the right side is $\sum_i (-1)^i \binom{i}{\alpha} \binom{k}{q-i}$. Next, $q-\alpha > k - \alpha- 1$, so the coefficient of $z^{q-\alpha}$ on the right side is 0. This proves our claim.

    We conclude that $u_i=0$ for $i < q-k$ and that $u_i = \binom{k}{q-i} u_q$ for $q-k \le i \le q$.

    Now suppose that $n=k$. Then we can apply M\"obius inversion to the $2^n-1$ equations \eqref{subeqn:binom-id} to conclude that $\bH^k_{\lambda,T}(x_1,\dots,x_n) =\bH^k_{\lambda,[2n]_{\rm odd}}(x_1,\dots,x_n)$ for all $T \subseteq [2n]_{\rm odd}$.
\end{proof}

\begin{remark}
  In general, we have $\sum_{i=0}^{k-1} \binom{q}{i}$ many equations of the form \eqref{subeqn:binom-id}. When $k=n$ or $k=1$, we solved this system explicitly, but in general, we collapsed it to $k$ equations and found a solution there. In all cases, we partitioned the terms $\bH^k_{\lambda,T}(x_1,\dots,x_n)$ into blocks and found explicit identities amongst their sums. We might ask for more refined partitions outside the cases $k=n$ and $k=1$, see the next example.
\end{remark}

\begin{example}
  Let $n=4$ and $k=2$. For simplicity, write $h_T$ in place of $\bH^2_{\lambda,T}(x_1,\dots,x_4)$.

  Then $h_\emptyset=0$ and \eqref{subeqn:binom-id} gives these 4 equations
  \begin{align*}
    h_{1} + h_{3} + h_{5} - h_{1,3} - h_{1,5} - h_{3,5} + h_{1,3,5} = 0,\\
h_{1} - h_{1,3} - h_{1,5} + h_{1,3,5} = 0,\\
h_{3} - h_{1,3} - h_{3,5} + h_{1,3,5} = 0,\\
h_{5} - h_{1,5} - h_{3,5} + h_{1,3,5} = 0.
  \end{align*}

  Subtracting each of the last three from the first gives
$h_{3} + h_{5} = h_{3,5}$, 
$h_{1} + h_{5} = h_{1,5}$, and
$h_{1} + h_{3} = h_{1,3}$.
We can combine these to find:
\[
h_{1,3,5} = h_{5} + h_{1,3} = h_{3} + h_{1,5} = h_{1} + h_{3,5}. \qedhere
\]
\end{example}

\subsection{Factorization}

We record the explicit weights \(\mu\) and \(\nu\) and the resulting product formula.

\begin{proposition}
  Assume $k \in \{ n-1 , n \}$.  Define
  \begin{align*}
  \mu &= \left( \lambda_{1},\ \lambda_{3}-1,\ \dots,\ \lambda_{2n-1} -n+1 \right)\\
  \nu &= \left( \lambda_{2} + n-1,\ \lambda_{4}+n-2,\ \dots,\ \lambda_{2k} + n-k \right).
  \end{align*}
  Then
  \[
    \bH_{\lambda, [n+k]_{\rm odd}}^k (x_1, \dots , x_n) = s_\mu (x_1 , \dots, x_n) s_\nu (x_1 , \dots , x_k).
  \]
In particular, $\bH_{\lambda}^k (x_1, \dots , x_n;1) = 2^k s_\mu (x_1 , \dots, x_n) s_\nu (x_1 , \dots , x_k)$.
\end{proposition}

\begin{proof}
  By Proposition~\ref{prop:equi-A}, we have
  \begin{align*}
    \bH^k_{\lambda}(x_1,\dots,x_n;1) &= (-1)^{\binom{n}{2}} 2^k (x_1\cdots x_n)^{-k} \frac{\det A'_\lambda(x;1)}{a_{\delta_n}(x_1,\dots,x_n) a_{\delta_k}(x_1,\dots,x_k)}\\
    &= 2^k s_{\beta^1_\lambda([n+k]_{\rm odd})} (x_1,\dots,x_n) s_{\beta^2_\lambda([n+k]_{\rm even})}(x_1,\dots,x_k).
  \end{align*}
  By definition, we have $\mu = \beta^1_\lambda([n+k]_{\rm odd})$ and $\nu = \beta^2_\lambda([n+k]_{\rm even})$.
\end{proof}

\begin{example}
Take $\lambda=(0,0,0,0)$, so $\lambda+\delta=(3,2,1,0)$. With variables $x_1,x_2$,
\[
A'_\lambda(x_1,x_2;1)=
\begin{bmatrix}
x_1^3 & -x_1^2 & x_1 & -1 \\
x_2^3 & -x_2^2 & x_2 & -1 \\
0 & x_1^2 & 0 & 1 \\
0 & x_2^2 & 0 & 1 \\
\end{bmatrix}.
\]
Expanding the determinant along the first two rows, the only nonzero contribution uses columns $\{1,3\}$; up to sign, after dividing by $a_\delta(x_1,x_2)^2$ this gives
\[
s_{(2,1)}(x_1,x_2)\cdot s_{(1,0)}(x_1,x_2)
= (x_1x_2)\cdot s_{(1,0)}(x_1,x_2)^2.
\]

For $n=3$, $k=2$, $\lambda=(2,1,0,0,0)$,
\[
A'_\lambda(x_1,x_2,x_3;1)=
\begin{bmatrix}
x_1^6 & -x_1^4 & x_1^2 & -x_1 & 1 \\
x_2^6 & -x_2^4 & x_2^2 & -x_2 & 1 \\
x_3^6 & -x_3^4 & x_3^2 & -x_3 & 1 \\
0     & x_1^4 & 0     & x_1 & 0 \\
0     & x_2^4 & 0     & x_2 & 0
\end{bmatrix}.
\]
If instead $k=1$,
\[
A'_\lambda(x_1,x_2,x_3;1)=
\begin{bmatrix}
x_1^5 & -x_1^3 & x_1 & -1  \\
x_2^5 & -x_2^3 & x_2 & -1  \\
x_3^5 & -x_3^3 & x_3 & -1  \\
0     & x_1^3 & 0   & 1 
\end{bmatrix},
\]
so the determinant decomposes as a sum of two products of minors (corresponding to choosing column $2$ or $4$ for the bottom row).
\end{example}

\section{Type BCD Combinatorial Identities}\label{sec:typeBCDidentities}

\subsection{Setup}
Pick $\eps \in \{0,1\}$. Set $m = 2n + \eps$. Similarly, pick $\gamma \in \{0,\frac12, 1\}$ and set $\gamma' = \lfloor \gamma \rfloor$. Throughout this section, we set
\[
  \delta = (m,m-1,\dots,1) - (\gamma, \gamma, \dots, \gamma).
\]
In particular, the $j$th entry of $\delta$ is $m+1-j-\gamma$.

Let $[m] = \{1,\dots,m\}$. For $S \subseteq [m]$, define
\[
  \Sigma(S) = \sum_{s \in S} s, \qquad \rank(S) = \Sigma(S) - \gamma'|S|.
\]
Recall that if $S=\{s_1,\dots,s_r\}$, we defined $S^\inv = \{m+1-s_r, \dots, m+1-s_1\}$ and $S^c = [m] \setminus S$.

Let $\lambda \in \bZ^m \cup (\tfrac12 + \bZ)^m$. For $S \subseteq [m]$, let $\iota_\lambda(S)$ be the result of negating the entries of $\lambda + \delta$ in positions indexed by $S$; let $\beta_\lambda(S)$ be the result of sorting $\iota_\lambda(S)$ in weakly decreasing order and then subtracting $\delta$. For $\alpha=(\alpha_1,\dots,\alpha_m)$ set $\alpha^{\rm op}:=(\alpha_m,\dots,\alpha_1)$ and $-\alpha^{\rm op} = (-\alpha_m,\dots,-\alpha_1)$.

\begin{example}
  Let \( m = 4 \) and \( \gamma = 0 \), so $\delta = (4,3,2,1)$. Fix the tuple $\lambda = (1,0,0,0)$, so that $\lambda + \delta = (5,3,2,1)$.
  We compute \( \iota_\lambda ( S) \) and \( \beta_\lambda(S) = \text{sort}(\iota_\lambda ( S)) - \delta \) for various \( S \subseteq [4] \):

\[
  \begin{array}{c|c|c|c|c}
    S & \iota_\lambda(S) &  \text{sort}(\iota_\lambda(S)) & \beta_\lambda(S) & -\beta_\lambda(S)^{\rm op}\\
    \hline
  \{1,3\} &      (-5,3,-2,1)     &(3,1,-2,-5)     &      (-1, -2, -4, -6)     &       (6,4,2,1)\\
    \{2,4\}  &     (5,-3,2,-1)     & (5,2,-1,-3)     & (1,-1,-3,-4)    & (4,3,1,-1)\\
 \{3,4\}     & (5,3,-2,-1)     & (5,3,-1,-2)     & (1,0,-3,-3)     &  (3,3,0,-1)\\
 \{2,3\}     & (5,-3,-2,1)     & (5,1,-2,-3)     & (1,-2,-4,-4)      &  (4,4,2,-1)
\end{array} \qedhere
\]
\end{example}

Sorting parity depends only on \(m\), \(|S|\), and \(\Sigma(S)\); we record this fact for later determinant expansions.

\begin{lemma} \label{lem:C-sign}
    The sign of the permutation that sorts $\iota_\lambda(S)$ is $(-1)^{m|S| + \Sigma(S)}$. In particular, the sign of the permutation that sorts $\iota_\lambda(S^{\inv})$ is $(-1)^{|S|+\Sigma(S)}$.
  \end{lemma}

  \begin{proof}
    Write $S = \{s_1 < \cdots < s_k\}$. Then the permutation that sorts $\iota_\lambda(S)$ moves position $s_i$ to position $m+1-i$ and keeps all other positions in relative order. In particular, the number of inversions is $\sum_{i=1}^k (m-s_i) \equiv m|S| + \Sigma(S) \pmod 2$.

    The second statement follows since $m|S^\inv| + \Sigma(S^\inv) = m|S| + (m+1)|S| - \Sigma(S) \equiv |S| + \Sigma(S) \pmod 2$.
  \end{proof}

Now fix $\lambda \in \bZ^m \cup (\tfrac12 + \bZ)^m$ and define
\begin{align*}
  \bH_\lambda(x_1,\dots,x_m;t)
  &:= \sum_{S \subseteq [m]} s_{\beta_\lambda(S)}(x_1,\dots,x_m) \cdot t^{\rank(S^\inv)}\\
  &= \sum_{S \subseteq [m]} \frac{a_{\beta_\lambda(S) + \delta}(x_1,\dots,x_m)}{a_\delta(x_1,\dots,x_m)} \cdot t^{\rank(S^\inv)}\\
  &= \sum_{S \subseteq [m]} (-1)^{m|S| + \Sigma(S)} \cdot \frac{a_{\iota_\lambda(S)}(x_1,\dots,x_m)}{a_\delta(x_1,\dots,x_m)} \cdot t^{\rank(S^\inv)},
\end{align*}
where the last line uses Lemma~\ref{lem:C-sign}.
Note that
\[
  m|S|+\Sigma(S) \equiv \eps|S| + |S_{\rm odd}| \equiv |S_{1-\eps}| \pmod 2.
\]
We also define a variation of this sum as follows. Given $T \subseteq [m]_{1-\eps}$, set
\[
  \bH_{\lambda, T}(x_1,\dots,x_m) = \sum_{\substack{S \subseteq [m]\\ S_{1-\eps} = T}} s_{\beta_\lambda(S)}(x_1,\dots,x_m) = (-1)^{|T|} \sum_{\substack{S \subseteq [m] \\ S_{1-\eps} = T}} \frac{a_{\iota_\lambda(S)}(x_1,\dots,x_m)}{a_\delta(x_1,\dots,x_m)}.
\]

Throughout this section, let $y_1,\dots,y_n$ be variables and set $y_{n+1}=1$. Define $z_1,\dots,z_m$ as follows: for $i=1,\dots,n$, set $z_i=y_i$ and $z_{i+n}=y_i^{-1}$, and if $m$ is odd, define $z_m = 1$. Alternatively:
\[
(z_1,\dots,z_m)=
\begin{cases}
(y_1,\dots,y_n,\,y_1^{-1},\dots,y_n^{-1}) & (m=2n),\\[2pt]
(y_1,\dots,y_n,\,y_1^{-1},\dots,y_n^{-1},\,1) & (m=2n+1).
\end{cases}
\]
The goal of this section is to prove the following analogous combinatorial properties as in the type A setting:
\begin{enumerate}
    \item[\textbf{Determinantal Form:}] The polynomial $\bH_\lambda (z_1 , \dots , z_m ; t)$ may be written as the determinant of a single matrix, and is divisible by $(t+1)^n$. 
    \item[\textbf{Equidistribution:}] The terms appearing in the sum $\bH_\lambda (z_1 , \dots , z_m ; 1)$ are ``equidistributed'' in a very strong sense: for \emph{any} two subsets $T , T' \subseteq [m]_{1 - \eps}$, we have
    \[ \bH_{\lambda, T}(z_1,\dots,z_m) = \bH_{\lambda, T'}(z_1,\dots,z_m).  \]
    As a consequence, $\bH_\lambda (z_1 , \dots , z_m ; 1) = 2^n \bH_{\lambda, \varnothing} (z_1 , \dots , z_m )$.
    \item[\textbf{Factorizability:}] For any choice of $\lambda$, there exist tuples $\mu$ and $\nu$ determined by $\lambda$ such that
    \[ \bH_{\lambda, \varnothing} (z_1 , \dots , z_m) = s^\rC_\mu (y_1 ,\dots , y_n) s^\rD_\nu (y_1 , \dots , y_{n + \eps}).\]
\end{enumerate}

\subsection{Determinantal form}

Define $m \times m$ matrices $M_\lambda(y;t)$ and $M'_\lambda(y;t)$ with entries
\begin{align*}
  M_\lambda(y;t)_{i,j} &= z_i^{\lambda_j + \delta_j}  + (-1)^{m+j}\, t^{\,m+1-j-\gamma'}\, z_i^{-(\lambda_j + \delta_j)},\\[2pt]
  M'_\lambda(y;t)_{i,j} &= \begin{cases}
    \Big(\sum_{r=0}^{\,m-j-\gamma'} (-t)^r\Big)\, \Big(z_i^{\lambda_j + \delta_j}  + (-1)^{\gamma'} z_i^{-(\lambda_j +
      \delta_j)}\Big) & \text{if } 1 \le i \le n,\\[4pt]
    z_i^{\lambda_j + \delta_j}  + (-1)^{m+j}\, t^{\,m+1-j-\gamma'}\, z_i^{-(\lambda_j + \delta_j)} & \text{if } i>n,
  \end{cases}
\end{align*}
with the convention that $\sum_{r=0}^{N} (\cdots )=0$ for $N < 0$. We record the determinantal form and its \((1+t)^n\) divisibility. 

\begin{proposition}\label{prop:typeBCDdetForm}
  We have
  \begin{align*}
    a_\delta(z_1,\dots,z_m) \bH_\lambda(z_1,\dots,z_m;t) = \det M_\lambda(y;t) = (1+t)^n \det M'_\lambda(y;t).
  \end{align*}
\end{proposition}

\begin{proof}
For the first equality, expand $\det M_\lambda(y;t)$ using multilinearity in the columns: for each column $j$, choose either the first or second summand from $M_\lambda(\cdot)_{i,j}$. Choosing the second summand exactly for $j\in S$ contributes
\[
\prod_{j\in S}\! \big((-1)^{m+j} t^{\,m+1-j-\gamma'}\big)\; \det\big(z_i^{\pm(\lambda_j+\delta_j)}\big)_{i,j},
\]
and sorting the resulting columns from $\iota_\lambda(S)$ to $\beta_\lambda(S)+\delta$ contributes the sign $(-1)^{m|S|+\Sigma(S)}$ by Lemma~\ref{lem:C-sign}. Since
\[
\sum_{j\in S}\!\big(m+1-j-\gamma'\big)=\Sigma(S^\inv)-\gamma'|S|=\rank(S^\inv),
\]
we obtain the displayed sum.

For the second equality, for $i=1,\dots,n$, add $(-1)^{\gamma'}$ times row $n+i$ of $M_\lambda$ to row $i$. Then factor out $1+t$ from each of rows $1,\dots,n$.
\end{proof}

Since the notation is dense, we spell this out explicitly in two different cases, which should clarify exactly what is happening with these determinants.

\begin{example}
  When $m = 2n=6$ and $\gamma = 0$ (take $\lambda=0$ for illustration), we have
    \[
M_\lambda(y_1, y_2, y_3; t) =
\begin{bmatrix}
y_1^6 - t^6 y_1^{-6} & y_1^5 + t^5 y_1^{-5} & y_1^4 - t^4 y_1^{-4} & y_1^3 + t^3 y_1^{-3} & y_1^2 - t^2 y_1^{-2} & y_1 + t y_1^{-1} \\[0.5em]
y_2^6 - t^6 y_2^{-6} & y_2^5 + t^5 y_2^{-5} & y_2^4 - t^4 y_2^{-4} & y_2^3 + t^3 y_2^{-3} & y_2^2 - t^2 y_2^{-2} & y_2 + t y_2^{-1} \\[0.5em]
y_3^6 - t^6 y_3^{-6} & y_3^5 + t^5 y_3^{-5} & y_3^4 - t^4 y_3^{-4} & y_3^3 + t^3 y_3^{-3} & y_3^2 - t^2 y_3^{-2} & y_3 + t y_3^{-1} \\[0.5em]
y_1^{-6} - t^6 y_1^6 & y_1^{-5} + t^5 y_1^5 & y_1^{-4} - t^4 y_1^4 & y_1^{-3} + t^3 y_1^3 & y_1^{-2} - t^2 y_1^2 & y_1^{-1} + t y_1 \\[0.5em]
y_2^{-6} - t^6 y_2^6 & y_2^{-5} + t^5 y_2^5 & y_2^{-4} - t^4 y_2^4 & y_2^{-3} + t^3 y_2^3 & y_2^{-2} - t^2 y_2^2 & y_2^{-1} + t y_2 \\[0.5em]
y_3^{-6} - t^6 y_3^6 & y_3^{-5} + t^5 y_3^5 & y_3^{-4} - t^4 y_3^4 & y_3^{-3} + t^3 y_3^3 & y_3^{-2} - t^2 y_3^2 & y_3^{-1} + t y_3
\end{bmatrix}
.
\]
Adding rows $4,5,6$ to rows $1,2,3$, respectively, gives:
{\tiny \[
    \begin{bmatrix}
(1 - t^6)(y_1^6 + y_1^{-6}) & (1 + t^5)(y_1^5 + y_1^{-5}) & (1 - t^4)(y_1^4 + y_1^{-4}) & (1 + t^3)(y_1^3 + y_1^{-3}) & (1 - t^2)(y_1^2 + y_1^{-2}) & (1 + t)(y_1 + y_1^{-1}) \\[0.5em]
(1 - t^6)(y_2^6 + y_2^{-6}) & (1 + t^5)(y_2^5 + y_2^{-5}) & (1 - t^4)(y_2^4 + y_2^{-4}) & (1 + t^3)(y_2^3 + y_2^{-3}) & (1 - t^2)(y_2^2 + y_2^{-2}) & (1 + t)(y_2 + y_2^{-1}) \\[0.5em]
(1 - t^6)(y_3^6 + y_3^{-6}) & (1 + t^5)(y_3^5 + y_3^{-5}) & (1 - t^4)(y_3^4 + y_3^{-4}) & (1 + t^3)(y_3^3 + y_3^{-3}) & (1 - t^2)(y_3^2 + y_3^{-2}) & (1 + t)(y_3 + y_3^{-1})\\[0.5em]
y_1^6 - t^6 y_1^{-6} & y_1^5 + t^5 y_1^{-5} & y_1^4 - t^4 y_1^{-4} & y_1^3 + t^3 y_1^{-3} & y_1^2 - t^2 y_1^{-2} & y_1 + t y_1^{-1} \\[0.5em]
y_2^6 - t^6 y_2^{-6} & y_2^5 + t^5 y_2^{-5} & y_2^4 - t^4 y_2^{-4} & y_2^3 + t^3 y_2^{-3} & y_2^2 - t^2 y_2^{-2} & y_2 + t y_2^{-1} \\[0.5em]
y_3^6 - t^6 y_3^{-6} & y_3^5 + t^5 y_3^{-5} & y_3^4 - t^4 y_3^{-4} & y_3^3 + t^3 y_3^{-3} & y_3^2 - t^2 y_3^{-2} & y_3 + t y_3^{-1} 
\end{bmatrix}
\]}
If we instead have $m = 2n+1 = 5$ and $\gamma = 1$, then
    {\footnotesize
\[
M_\lambda(y_1, y_2; t) =
\begin{bmatrix}
y_1^{\lambda_1+4} + t^4 y_1^{-(\lambda_1+4)} &
y_1^{\lambda_2+3} - t^3 y_1^{-(\lambda_2+3)} &
y_1^{\lambda_3+2} + t^2 y_1^{-(\lambda_3+2)} &
y_1^{\lambda_4+1} - t y_1^{-(\lambda_4+1)} &
y_1^{\lambda_5} + y_1^{-\lambda_5} \\[6pt]
y_2^{\lambda_1+4} + t^4 y_2^{-(\lambda_1+4)} &
y_2^{\lambda_2+3} - t^3 y_2^{-(\lambda_2+3)} &
y_2^{\lambda_3+2} + t^2 y_2^{-(\lambda_3+2)} &
y_2^{\lambda_4+1} - t y_2^{-(\lambda_4+1)} &
y_2^{\lambda_5} + y_2^{-\lambda_5} \\[6pt]
y_1^{-(\lambda_1+4)} + t^4 y_1^{\lambda_1+4} &
y_1^{-(\lambda_2+3)} - t^3 y_1^{\lambda_2+3} &
y_1^{-(\lambda_3+2)} + t^2 y_1^{\lambda_3+2} &
y_1^{-(\lambda_4+1)} - t y_1^{\lambda_4+1} &
y_1^{-\lambda_5} + y_1^{\lambda_5} \\[6pt]
y_2^{-(\lambda_1+4)} + t^4 y_2^{\lambda_1+4} &
y_2^{-(\lambda_2+3)} - t^3 y_2^{\lambda_2+3} &
y_2^{-(\lambda_3+2)} + t^2 y_2^{\lambda_3+2} &
y_2^{-(\lambda_4+1)} - t y_2^{\lambda_4+1} &
y_2^{-\lambda_5} + y_2^{\lambda_5} \\[6pt]
1 + t^4 & 1 - t^3 & 1 + t^2 & 1 - t & 2
\end{bmatrix}.
\]
}
Subtracting rows $3,4$ from rows $1,2$, respectively, yields the following matrix:
{\tiny \[
\begin{bmatrix}
(1 - t^4)(y_1^{\lambda_1+4} - y_1^{-(\lambda_1+4)}) &
(1 + t^3)(y_1^{\lambda_2+3} - y_1^{-(\lambda_2+3)}) &
(1 - t^2)(y_1^{\lambda_3+2} - y_1^{-(\lambda_3+2)}) &
(1 + t)(y_1^{\lambda_4+1} - y_1^{-(\lambda_4+1)}) &
0 \\[6pt]

(1 - t^4)(y_2^{\lambda_1+4} - y_2^{-(\lambda_1+4)}) &
(1 + t^3)(y_2^{\lambda_2+3} - y_2^{-(\lambda_2+3)}) &
(1 - t^2)(y_2^{\lambda_3+2} - y_2^{-(\lambda_3+2)}) &
(1 + t)(y_2^{\lambda_4+1} - y_2^{-(\lambda_4+1)}) &
0 \\[6pt]

y_1^{-(\lambda_1+4)} + t^4 y_1^{\lambda_1+4} &
y_1^{-(\lambda_2+3)} - t^3 y_1^{\lambda_2+3} &
y_1^{-(\lambda_3+2)} + t^2 y_1^{\lambda_3+2} &
y_1^{-(\lambda_4+1)} - t y_1^{\lambda_4+1} &
y_1^{-\lambda_5} + y_1^{\lambda_5} \\[6pt]

y_2^{-(\lambda_1+4)} + t^4 y_2^{\lambda_1+4} &
y_2^{-(\lambda_2+3)} - t^3 y_2^{\lambda_2+3} &
y_2^{-(\lambda_3+2)} + t^2 y_2^{\lambda_3+2} &
y_2^{-(\lambda_4+1)} - t y_2^{\lambda_4+1} &
y_2^{-\lambda_5} + y_2^{\lambda_5} \\[6pt]

1 + t^4 & 1 - t^3 & 1 + t^2 & 1 - t & 2
\end{bmatrix}. 
\]}
\end{example}

\subsection{Equidistribution}

\begin{proposition}
  For any $T, T' \subseteq [m]_{1-\eps}$, we have
  \[
    \bH_{\lambda, T}(z_1,\dots,z_m) = \bH_{\lambda,T'}(z_1,\dots,z_m).
  \]
  In particular, $\bH_\lambda(z_1,\dots,z_m;1) = 2^n \bH_{\lambda, \varnothing}(z_1,\dots,z_m)$.
\end{proposition}

\begin{proof}
This is equivalent to showing that for all nonempty $\bT \subseteq [m]_{1-\eps}$, we have
\[
  \sum_{T \subseteq \bT} (-1)^{|T|} \bH_{\lambda, T}(z_1,\dots,z_m) = 0.
\]
If we multiply by $a_\delta(z_1,\dots,z_m)$, this is equivalent to showing that
\[
  \sum_{\substack{S \subseteq [m] \\ S_{1-\eps} \subseteq \bT}} a_{\iota_\lambda(S)}(z_1,\dots,z_m) = 0.
\]
The left-hand side is the determinant of the $m \times m$ matrix $M(\bT)$ defined by:
\[
  M(\bT)_{i,j} = \begin{cases}
    z_i^{\lambda_j + \delta_j} & \text{if $j \notin \bT \cup [m]_{\eps}$}\\
    z_i^{\lambda_j + \delta_j} + z_i^{-(\lambda_j + \delta_j)} & \text{if $j \in \bT \cup [m]_{\eps}$}
  \end{cases}.
\]
For each $i=1,\dots,n$ and $j \in \bT \cup [m]_\eps$, we have $M(\bT)_{i,j} = M(\bT)_{i+n,j}$. So if we subtract row $i+n$ from row $i$ for each $i=1,\dots,n$, then in the resulting matrix, the first $n$ rows has $|\bT \cup [m]_\eps| = |\bT| + n+\eps$ columns which are identically 0. In particular, it has rank $\le m - |\bT| - n - \eps = n - |\bT|$. Since $|\bT|>0$, we conclude that $\det M(\bT)=0$.
\end{proof}

\subsection{Factorization}

Now we express $\det M_\lambda'(y;1)$ as a product of two determinants: one of size $n \times n$ and the other of size $(n+\eps) \times (n+\eps)$.  Recall that our convention is that $y_{n+1}=1$.

\begin{proposition}
  Define
  \begin{align*}
    \mu &= \big(\lambda_{1+\varepsilon} + n - \gamma,\ \lambda_{3+\varepsilon} + n-1 - \gamma,\ \dots,\ \lambda_{2n-1+\varepsilon} + 1 - \gamma\big),\\
    \nu &= \big(\lambda_{2-\varepsilon} + n + \varepsilon - \gamma,\ \lambda_{4-\varepsilon} + n + \varepsilon - 1 - \gamma,\ \dots,\ \lambda_{2n+\varepsilon} + 1 - \gamma\big).
  \end{align*}
  Then
  \[
  \bH_{\lambda, \varnothing} (z_1 , \dots , z_m ) \;=\; s_\mu^\rC (y_1 , \dots , y_n)\, s_\nu^\rD (y_1 , \dots , y_{n+\varepsilon}),
  \]
  and in particular 
  $\bH_\lambda (z_1 , \dots , z_m ; 1) \;=\; 2^n\, s_\mu^\rC (y_1 , \dots , y_n)\, s_\nu^\rD (y_1 , \dots , y_{n+\varepsilon})$.
\end{proposition}

\begin{proof}
{\bf Case 1.} Suppose that $\eps+\gamma'$ is even. In the first $n$ rows of $M'_\lambda(y;1)$, the odd-indexed columns are identically 0. Denote the submatrix consisting of the even-indexed columns by
\[
  T'_\lambda = ( y_i^{\lambda_{2j} + \delta_{2j}} + (-1)^{\gamma'} y_i^{-(\lambda_{2j} + \delta_{2j})})_{i,j=1,\dots,n},
\]
and the bottom submatrix consisting of rows $n+1,\dots,2n+\eps$ and the odd-indexed columns by
\[
  B'_\lambda = ( y_i^{- (\lambda_{2j-1} + \delta_{2j-1})} - (-1)^{\eps} y_i^{\lambda_{2j-1} + \delta_{2j-1}})_{i,j=1,\dots,n+\eps}.
\]
Then we have
\[
  \det M_\lambda(y;1) = (-1)^{\binom{n+1}{2}} 2^n \det T'_\lambda \cdot \det B'_\lambda,
\]
by the two equalities
\[
\det M'_\lambda(y;1)=(-1)^{\binom{n+1}{2}} \det T'_\lambda\det B'_\lambda,\qquad
\det M_\lambda(y;1)=2^n\det M'_\lambda(y;1).
\]
Now we identify these determinants. If $\eps = \gamma' = 0$, then $m=2n$ and
\begin{align*}
  \det T'_\lambda = 2d_{(\lambda+\delta)_{\rm even}}(y_1,\dots,y_n), \qquad \det B'_\lambda = (-1)^n c_{(\lambda+\delta)_{\rm odd}}(y_1,\dots,y_n).
\end{align*}

If instead $\eps = \gamma' = 1$, then $m=2n+1$, $\gamma=1$, and
\begin{align*}
  \det T'_\lambda = c_{(\lambda+ \delta)_{\rm even}}(y_1,\dots,y_n), \qquad \det B'_\lambda = 2d_{(\lambda + \delta)_{\rm odd}}(y_1,\dots,y_n,1).
\end{align*}

{\bf Case 2.} Suppose that $\eps+\gamma'$ is odd. Consider the first $n$ rows of $M'_\lambda(y;1)$. The even-indexed columns are identically 0. Denote the $n \times n$ submatrix consisting of the even-indexed columns and rows $n+1,\dots,2n$ by
\[
  B'_\lambda = (
      y_i^{- (\lambda_{2j} + \delta_{2j})} + (-1)^{\eps} y_i^{\lambda_{2j} + \delta_{2j}})_{i,j=1,\dots,n}.
\]
Denote the complementary $(n+\eps) \times (n+\eps)$ submatrix consisting of the odd-indexed columns in rows $1,\dots,n$ (and $2n+1$ if $\eps=1$) by $T'_\lambda$:
\[
  (T'_\lambda)_{i,j} = \begin{cases}
    y_i^{\lambda_{2j-1} + \delta_{2j-1}} + (-1)^{\gamma'} y_i^{-(\lambda_{2j-1} + \delta_{2j-1})} & \text{if $1\le i \le n$}\\
    2 & \text{if $\eps=1$ and $i=2n+1$}
    \end{cases}.
\]

Then we have
\[
  \det M_\lambda(y;1) = (-1)^{\binom{n}{2}} 2^n \det T'_\lambda \cdot \det B'_\lambda.
\]
Now we identify these determinants.
If $\eps=0$ and $\gamma'=1$, then $m=2n$ and
\[
  \det B'_\lambda = 2 d_{(\lambda+\delta)_{\rm even}}(y_1,\dots,y_n),
  \qquad
  \det T'_\lambda = c_{(\lambda+\delta)_{\rm odd}}(y_1,\dots,y_n).
\]

Finally, if $\eps = 1$, $\gamma' = 0$, then $m=2n+1$ and
\[
  \det B'_\lambda = (-1)^n c_{(\lambda+\delta)_{\rm even}}(y_1,\dots,y_n), \qquad \det T'_\lambda = 2 d_{(\lambda+ \delta)_{\rm odd}}(y_1,\dots,y_n,1). \qedhere
\]
\end{proof}

\part{Algebraic Reinterpretation} \label{part:2}

\section{Kostant's theorem for Lie algebra homology}\label{subsec:KostantThm}

We recall Kostant’s decomposition of Lie algebra homology and fix conventions for the dot action and parabolic data.

\subsection{Setup}
Let \( \mathfrak{g} \) be a complex semisimple Lie algebra with Cartan subalgebra \( \mathfrak{h} \subset \mathfrak{g} \). Fix a choice of positive roots \( \Delta^+ \subset \Delta \). Let \( \mathfrak{p} \subseteq \mathfrak{g} \) be a standard parabolic subalgebra, and write $\mathfrak{p} = \mathfrak{l} \oplus \fn$, where \( \mathfrak{l} \) is the Levi subalgebra and \( \mathfrak{n} \) is the nilpotent radical of \( \mathfrak{p} \). Let $\fp_- = \fl \oplus \fn_-$ denote the opposite parabolic, with $\fn_-$ its nilpotent radical.

We let \( W \) denote the Weyl group of \( \mathfrak{g} \), and \( W_\fp \subseteq W \) denote the Weyl group of \( \mathfrak{l} \). Define \( W^\fp \subseteq W \) to be the set of minimal-length coset representatives for \( W / W_{\fp} \). Concretely, the minimal-length coset representatives in $W^\fp$ are precisely those $w \in W$ such that $w^{-1} \bullet 0$ (with the dot action defined below) is a dominant weight when restricted to the Levi subalgebra $\fl$ (with respect to $\Delta^+(\fl)$).

Next, define $\rho \in \fh^*$ to be the half-sum of positive roots (equivalently, the sum of the fundamental weights) of $\fg$ and the \defi{dot action}
\[
  w \bullet \lambda := w(\lambda + \rho) - \rho.
\]
Let $L^\fg_\lambda$ denote the irreducible $\fg$-representation of highest weight $\lambda$, and likewise for $L^\fl_{\mu}$. Kostant’s theorem decomposes \(\rH_i(\fn_-;L^\fg_\lambda)\) as a direct sum of \(\fl\)-highest weight modules indexed by elements of \(W^\fp\) of length \(i\).

\begin{theorem}[\cite{kostant}]\label{thm:KostantsTheorem}
With notation as above, we have an $\fl$-equivariant decomposition:
\[
\rH_i(\fn_-; L^\mathfrak{g}_\lambda) \cong \bigoplus_{\substack{w \in W^\mathfrak{p} \\ \ell(w) = i }} L^\fl_{w^{-1} \bullet \lambda}.
\]
\end{theorem}

Note that since $\rU (\fn_-)$ is a nonnegatively graded $\bC$-algebra, Kostant's theorem implies that there exists an $\fl$-equivariant minimal complex $C_\bullet$ of free $\rU (\fn_-)$-modules with terms
\[C_i = \bigoplus_{\substack{w \in W^\mathfrak{p} \\ \ell(w) = i }} \rU (\fn_-) \otimes L^\fl_{w^{-1} \bullet \lambda}.
\]
In fact, if the summands are viewed as parabolic Verma modules, then the statement can be upgraded to $\fg$-equivariance. These complexes are typically referred to as \defi{(parabolic) BGG resolutions}, and the structure of these objects is described in \cite{BGG1975,Lepowsky1977}.

For the remainder of this section we translate Kostant's theorem to each of the relevant types and deduce the relationship of the formulas derived in the previous sections with the graded characters of Lie algebra homology. We will also consider the case $\fg = \fgl_n$; this behaves almost exactly the same way as $\fsl_n$ except we will be careful to keep track of copies of the trace character.

A \defi{partition} is a tuple of nonnegative integers $\lambda = (\lambda_1 , \dots , \lambda_n)$ with $\lambda_1 \geq \cdots \geq \lambda_n$. Its \defi{length}, denoted $\ell(\lambda)$, is the number of nonzero $\lambda_i$. A \defi{Young diagram} of shape $\lambda$ is a left-justified set of boxes whose $i$th row has  $\lambda_i$ boxes. Given a partition $\lambda$, its \defi{transpose} $\lambda^T$ is the partition obtained by reading off the column lengths of the Young diagram of $\lambda$. In a formula: $\lambda^T_i = \# \{j \mid \lambda_j \ge i\}$. A partition $\lambda$ is \defi{self-conjugate} if $\lambda = \lambda^T$. We treat partitions as vectors, so we can add them and scale them, e.g., $2\lambda = (2\lambda_1,2\lambda_2,\dots)$. The notation $(a^b)$ denotes a sequence of $b$ copies of $a$.
  
For background on Schur functors, see \cite[\S 2]{weyman2003}, but note that the indexing used there is transposed from standard conventions, i.e., the Schur functor indexed by $\lambda$ there is what we would call $\bS_{\lambda^T}$. To be extra precise, our conventions are such that
    $\bS_{(d)} = S^d$ and $\bS_{(1^d)} = \bigwedge^d$.
    If $V$ is an $m$-dimensional space, then for any $d \ge 0$, we have $\bS_\lambda V \otimes L^{\otimes d} \cong \bS_{\lambda + (d,\dots,d)}(V)$ where $L$ is the trace character of $\fgl(V)$ (or determinant character if we use the group $\GL(V)$). Since $L^{-1}$ makes sense, we can use this identity to define $\bS_\lambda V$ when $\lambda$ is a weakly decreasing sequence of integers which has negative entries.

\section{Type A Generalized $\rho$-Decomposition}\label{sec:typeArhodecomp}

\subsection{Setup}

Let $V$ and $U$ be vector spaces of dimensions $n$ and $k$, respectively, and without loss of generality assume $n \geq k$. Fix the standard parabolic subalgebra $\fp$ with Levi factor \( \mathfrak{l} \cong \mathfrak{gl}_n \times \mathfrak{gl}_k \) induced by the decomposition
\[
  \fgl (V \oplus U) = \underbrace{V^* \otimes U}_{\fn_-} \oplus
  \underbrace{(\fgl(V) \oplus \fgl(U))}_\fl
\oplus \underbrace{U^* \otimes V}_{\fn}.
\]
In this section, we set up the relevant notation and state Kostant’s formula in the \(\fgl(V \oplus U)\) setting. The Weyl group is the symmetric group $\fS_{n+k}$ and the corresponding Weyl subgroup of $\fp$ is
\[
W_\fp = \fS_n \times \fS_k \subset \fS_{n+k}.
\]

The set \( W^\fp \subset \fS_{n+k} \) of minimal length coset representatives for \( W / W_\fp \) consists of those \( w \in \fS_{n+k} \) such that $w^{-1} \bullet 0 \text{ is dominant for } \mathfrak{l}$. Explicitly, writing \( w^{-1} \bullet 0 = (\mu_1, \dots, \mu_{n+k}) \), this translates to:
\begin{itemize}
  \item \( \mu_1 \ge \mu_2 \ge \dots \ge \mu_n \) (the first \( n \) entries are weakly decreasing), and
  \item \( \mu_{n+1} \ge \mu_{n+2} \ge \dots \ge \mu_{n+k} \) (the last \( k \) entries are weakly decreasing).
\end{itemize}
These are the Weyl group elements that contribute to Kostant's theorem for the Lie algebra homology of \( \fn_- \).

Recall that for a nonnegative integer $m$, we define \( [m] := \{1,2,\dots,m\} \), and let \( \binom{[n+k]}{n} \) denote the set of all \( n \)-element subsets of \( [n+k] \). We equip this set with a partial order called the \defi{Gale order} (also known as the \defi{componentwise order}), defined as follows.

  \medskip
  
\noindent
\begin{minipage}{0.75\textwidth}
\begin{definition}
Let \( A = \{a_1 < a_2 < \cdots < a_n\} \) and \( B = \{b_1 < b_2 < \cdots < b_n\} \) be two elements of \( \binom{[n+k]}{n} \). We say that
\[
A \leq_{\text{Gale}} B
\]
if and only if
\[
a_i \leq b_i \quad \text{for all } i = 1, \dots, n. \qedhere
\]
\end{definition}

The Gale order defines a ranked poset structure on \( \binom{[n+k]}{n} \) with minimal element \( \{1,2,\dots,n\} \) and maximal element \( \{k+1, \dots, n+k\} \) of rank $nk$. The rank of a general element $\{ a_1 , \dots , a_n \}$ is given by 
\[
  \rank \{a_1 , \dots , a_n \} = \sum_{i = 1}^n (a_i - i) = \Sigma(\{a_1,\dots,a_n\}) - \binom{n+1}{2},
\]
where $\Sigma(S) = \sum_{s \in S} s$.
On the right, we have drawn the Hasse diagram of the Gale order for $n=k=3$.

\end{minipage}
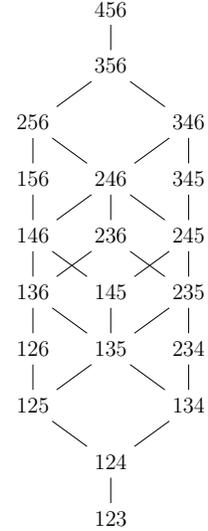
\hfill
\begin{minipage}{0.2\textwidth}
\adjustbox{scale=.7,center}{
  \begin{tikzcd}[row sep=1.2em, column sep=small]
	& 456 \\
	& 356 \\
	256 && 346 \\
	156 & 246 & 345 \\
	146 & 236 & 245 \\
	136 & 145 & 235 \\
	126 & 135 & 234 \\
	125 && 134 \\
	& 124 \\
	& 123
	\arrow[no head, from=1-2, to=2-2]
	\arrow[no head, from=2-2, to=3-1]
	\arrow[no head, from=2-2, to=3-3]
	\arrow[no head, from=3-1, to=4-1]
	\arrow[no head, from=3-1, to=4-2]
	\arrow[no head, from=3-3, to=4-3]
	\arrow[no head, from=4-1, to=5-1]
	\arrow[no head, from=4-2, to=3-3]
	\arrow[no head, from=4-2, to=5-1]
	\arrow[no head, from=4-2, to=5-2]
	\arrow[no head, from=4-2, to=5-3]
	\arrow[no head, from=4-3, to=5-3]
	\arrow[no head, from=5-1, to=6-1]
	\arrow[no head, from=5-1, to=6-2]
	\arrow[no head, from=5-2, to=6-1]
	\arrow[no head, from=5-2, to=6-3]
	\arrow[no head, from=5-3, to=6-2]
	\arrow[no head, from=5-3, to=6-3]
	\arrow[no head, from=6-1, to=7-1]
	\arrow[no head, from=6-1, to=7-2]
	\arrow[no head, from=6-2, to=7-2]
	\arrow[no head, from=6-3, to=7-2]
	\arrow[no head, from=6-3, to=7-3]
	\arrow[no head, from=7-1, to=8-1]
	\arrow[no head, from=7-2, to=8-1]
	\arrow[no head, from=7-2, to=8-3]
	\arrow[no head, from=7-3, to=8-3]
	\arrow[no head, from=8-1, to=9-2]
	\arrow[no head, from=8-3, to=9-2]
	\arrow[no head, from=9-2, to=10-2]
      \end{tikzcd}
      }
\end{minipage}

\begin{proposition}\label{prop:galeAndBruhat}
  $W^\fp$ is naturally in bijection with $\binom{[n+k]}{n}$ via the map
\[
S = \{ s_1 < \cdots < s_n \} \longmapsto w_S,
\]
where $w_S \in W^\fp$ sends positions $1,\dots,n$ to $s_1,\dots,s_n$ (in order), and positions $n+1,\dots,n+k$ to the complement $[n+k] \setminus S$ (also in order). Under this bijection,
\begin{enumerate}
\item the Bruhat order on $W^\fp$ corresponds to the Gale order on $\binom{[n+k]}{n}$:
\[
S \leq_{\textnormal{Gale}} T \quad \Longleftrightarrow \quad w_S \leq w_T \text{ in Bruhat order}. 
\]
\item Rank corresponds to length: $\rank(T) = \ell(w_T)$.
\end{enumerate}
\end{proposition}

The proof of Proposition \ref{prop:galeAndBruhat} is likely well-known, but we could not find a reference.

\begin{proof}
  Assume that $S = \{ s_1 < \cdots < s_n\}$ is covered by $T = \{ t_1 < \cdots < t_n \}$ with respect to the Gale ordering. By definition, this means that there is a distinguished integer $1 \leq i \leq n$ satisfying $s_j = t_j$ for $i \neq j$, and $t_i = s_i+1$. Then $w_T$ is obtained from $w_S$ by multiplying on the left by the simple transposition $(s_i, s_i+1)$. By definition, it follows that $w_T$ covers $w_S$ in the Bruhat ordering. The fact that the ranks are preserved is an immediate consequence.
\end{proof}

\subsection{Statement and Examples}

Using the notation from \S\ref{sec:A-setup}, define
    \[
      \bS_{\beta_\lambda(S)} (V,U) =
      \bS_{\beta^1_\lambda(S)}(V) \otimes \bS_{\beta^2_\lambda(S)}(U). 
    \]

    Here is Kostant's theorem specialized to our situation:

\begin{lemma} \label{lem:A-homology}
    For all $i \geq 0$ there is a $\fgl(V) \times \fgl(U)$-equivariant isomorphism
    \[
      \rH_i ( V^* \otimes U ; L^{\fgl (V \oplus U)}_\lambda ) = \bigoplus_{\substack{S \in \binom{[n+k]}{n} \\ \rank S = i}} \bS_{\beta_\lambda (S)}(V,U).
    \]
\end{lemma}

\begin{proof}
We employ Kostant's Theorem \ref{thm:KostantsTheorem} and reinterpret it through the combinatorics developed in the previous section.

By Proposition~\ref{prop:galeAndBruhat}, the set \(W^\fp\) is identified with
\(\binom{[n+k]}{n}\) via \(S\mapsto w_S\), and under this identification
\(\ell(w_S)=\rank(S)\). Thus it remains only to compute the \(\fl\)-highest weight
\(w_S^{-1}\bullet \lambda\).

Write
\[
  S=\{s_1<\cdots<s_n\},
  \qquad
  S^c=\{t_1<\cdots<t_k\}.
\]
By definition of \(w_S\), we have
\[
  w_S(1)=s_1,\dots,w_S(n)=s_n,
  \qquad
  w_S(n+1)=t_1,\dots,w_S(n+k)=t_k.
\]
Hence
\[
  w_S^{-1}(\lambda+\delta_{n+k})
  =
  \bigl((\lambda+\delta_{n+k})_{s_1},\dots,(\lambda+\delta_{n+k})_{s_n},
        (\lambda+\delta_{n+k})_{t_1},\dots,(\lambda+\delta_{n+k})_{t_k}\bigr).
\]
Subtracting \(\delta_{n+k}=(n+k-1,n+k-2,\dots,0)\), whose first \(n\) entries are
\(n+k-1,\dots,k\) and whose last \(k\) entries are \(k-1,\dots,0\), we get
\[
  w_S^{-1}\bullet \lambda
  =
  \bigl(\lambda_{s_1}-s_1+1,\dots,\lambda_{s_n}-s_n+n,
        \lambda_{t_1}+n-t_1+1,\dots,\lambda_{t_k}+n-t_k+k\bigr).
\]
Equivalently, the first \(n\) entries are
\[
  (\lambda+\delta_{n+k})|_S-\delta_n-(k,\dots,k)=\beta^1_\lambda(S),
\]
and the last \(k\) entries are
\[
  (\lambda+\delta_{n+k})|_{S^c}-\delta_k=\beta^2_\lambda(S).
\]
Thus we have an equality of \(\fl\)-highest weights
\[
  w_S^{-1}\bullet \lambda
  =   \bigl(\beta^1_\lambda(S),\beta^2_\lambda(S)\bigr). \qedhere
\]
\end{proof}

Notice that in the notation of Section \ref{sec:typeAidentities}, the above lemma implies that we have the equality
$$\bH_\lambda^k (x_1  , \dots , x_n ; t) = \sum_{i \geq 0} \ch \rH_i (V^* \otimes U ; L_\lambda^{\fgl (V \oplus U)}) \cdot  t^i.$$

Set \(q:= \lceil \frac{n+k}{2} \rceil =|[n+k]_{\rm odd}|\) (the number of odd indices in \([n+k]\)). We now deduce the following consequences for graded characters: divisibility, equidistribution, factorization at \(k\in\{n-1,n\}\), and a closed form for total dimension.

\begin{theorem} \label{thm:A-rho}
  Let $V$ and $U$ be vector spaces of dimensions $n$ and $k$, respectively, with $n \geq k$. Let $\lambda = (\lambda_1 , \dots , \lambda_{n+k}) \in \bZ^{n+k}$ with $\lambda_1 \geq \lambda_2 \geq \cdots \geq \lambda_{n+k}$.

  In what follows, we consider characters of the $n$-dimensional toral subalgebra $\ft \subset \fgl(V) \times \fgl(U)$ given by pairs of diagonal matrices of the form $(\diag(a_1,\dots, a_n), \diag(a_1,\dots,a_k))$.
  
    \begin{enumerate}
    \item The graded character $\sum_{i \geq 0} \ch \rH_i (\fn_- ; L_\lambda^{\fgl (V \oplus U)}) \cdot t^i$ is divisible by $(1+t)^k$. 
  \item For any integer $i \geq 0$ we have
    \[
      \sum_{\substack{S \in \binom{[n+k]}{n} \\ |S_{\rm odd}| = i}} \ch \bS_{\beta_\lambda (S)} (V,U) = \binom{k}{q-i} \sum_{\substack{S \in \binom{[n+k]}{n} \\ S_{\rm odd} = [n+k]_{\rm odd}}} \ch \bS_{\beta_\lambda (S)} (V,U).
      \]
  In particular,
  $$\sum_{i \geq 0} \ch \rH_i (\fn_- ; L_\lambda^{\fgl (V \oplus U)}) = 2^k \sum_{\substack{S \in \binom{[n+k]}{n} \\ S_{\rm odd} = [n+k]_{\rm odd}}} \ch \bS_{\beta_\lambda (S)} (V,U).$$
  \item When $k \in \{ n-1, n \}$, we have 
  \[\ch \rH_\bullet (\fn_- ; L_\lambda^{\fgl (V \oplus U)}) = 2^k s_{\rho^{\rm top}_\lambda} (x_1 , \dots , x_n) s_{\rho^{\rm bot}_\lambda} (x_1 , \dots , x_k),  \]
where
\begin{align*}
\rho^{\mathrm{top}}_\lambda &= \left( \lambda_{1},\ \lambda_{3}-1,\ \dots,\ \lambda_{2n-1} -n+1 \right), \\
  \rho^{\mathrm{bot}}_\lambda &= \left( \lambda_{2} + n-1,\ \lambda_{4}+n-2,\ \dots,\ \lambda_{2k} + n-k \right).                              
\end{align*}

\item When $k \in \{ n-1, n \}$, the \emph{total} dimension satisfies
  \begin{align*}
    \dim \rH_\bullet (V^* \otimes U; L_\lambda^{\fgl (V \oplus U)}) &= 2^{\dim (V^* \otimes U)} \cdot \prod_{\substack{1 \leq i < j \leq n+k\\ i\equiv j \pmod 2}} \left(1 + \frac{\lambda_{i} - \lambda_{j}}{j-i}\right).                                                                    
  \end{align*}
\end{enumerate}
  \end{theorem}

\begin{remark}\label{rk:2nCorJustification}
    The statement of Corollary \ref{cor:2n+2n-1 conj} is an immediate consequence of the above dimension formula, since if not all parts of $\lambda$ are equal (i.e., $\lambda$ is not a one-dimensional representation), then $\lambda_i - \lambda_{i+2} \geq 1$ for some $i$. This means a factor of $1 + 1/2 = 3/2$ appears as a factor in the product formula, which yields the desired result.

    Another point worth mentioning is that even though the product formula for the dimension can be stated for arbitrary $k$, this product does \emph{not} give the dimension $\dim \rH_\bullet (V^* \otimes U; L_\lambda^{\fgl (V \oplus U)})$ for general weights $\lambda$ when $k \notin \{ n-1,n \}$. 
\end{remark}

Note that the only thing that requires proof in Theorem \ref{thm:A-rho} is part (4), since the parts (1)--(3) are immediate consequences of the results proved in \S\ref{sec:typeAidentities}, so we delay the proof and illustrate the results with some examples first.

Recall that if a Lie algebra $\fg$ is abelian then $\rU (\fg) \cong \Sym (\fg)$ and there is an equality
$$\rH_\bullet (\fg ; E) = \Tor_\bullet^{\Sym (\fg) } (E , \bC),$$
where $E$ is any representation of $\fg$. Thus in the context of Kostant's theorem for our setting, the parabolic BGG resolutions that we are interested in are minimal free resolutions of finite length modules over a polynomial ring.

\begin{example}\label{ex:pureFreeRes}
    In the special case $k = 1$ (with notation as in Theorem \ref{thm:A-rho}), the parabolic BGG resolutions are resolutions over $\Sym(V^* \otimes U) \cong \Sym (V^*)$, and these recover the pure free resolutions constructed by Eisenbud--Fl\o ystad--Weyman in \cite{EFW}. This fact is mentioned in \cite[Remark 4.6]{FLS}, but we spell out the details explicitly here. For convenience, we reverse the order of $n$ and $k$ and assume $k=1$ (in other words we are looking at the parabolic decomposition induced by isolating the first node in the type A Dynkin diagram).

    Note that when $k = 1$ the Gale ordering is a total ordering, and Kostant's theorem implies that we will obtain resolutions over $\Sym (U)$ with only a single irreducible representation in each homological degree. To obtain a resolution with differentials of pure degree $e_1 , e_2 , e_3 , \dots , e_n$ (respectively), we may choose 
    $$\lambda = (\sum_{i=1}^n e_i - n, \sum_{i=2}^{n} e_i - (n-1) , \dots , e_n-1 , 0) \in \bZ^{n+1}.$$
    The dot action on $\lambda$ yields the sequence of partitions
    $\lambda = \lambda^{(0)}, \ \lambda^{(1)}, \dots ,\lambda^{(n)}$
    with
    $$\lambda^{(i)} = (\lambda_{i+1} - i, \underbrace{\lambda_1+1}_{= \lambda_2+e_1}, \underbrace{\lambda_2+1}_{=\lambda_3 + e_2} , \dots , \underbrace{\lambda_i+1}_{= \lambda_{i+1} + e_{i}}, \lambda_{i+2} , \dots , 0 ).$$
    Ignoring the first entry of each $\lambda^{(i)}$, we obtain the sequence of partitions $\alpha (\mathbf{d} , i)$ defined in \cite{EFW}; the first entry of $\lambda^{(i)}$ is simply recording the degree shift of the modules.
\end{example}

\begin{example}
Consider \(\dim V = 2\), \(\dim U = 2\), and the representation
\[
L_{(1,1,0,0)}^{\fgl (V \oplus U)} = \bigwedge^2(V \oplus U).
\]
Its free resolution over \( A = \Sym(V^* \otimes U) \) has the following terms, along with the corresponding subsets \( S \subset [4] \).
\begin{align*}
&\bS_{1,1} V \otimes A, && S = \{1,2\}, \\
&\bS_{1,-1} V \otimes \bS_2 U \otimes A(-2), && S = \{1,3\}, \\
&\bS_{0,-1} V \otimes \bS_3 U \otimes A(-3) 
\ \oplus \ \bS_{1,-2} V \otimes \bS_{2,1} U \otimes A(-3), 
&& \begin{cases}
S = \{2,3\}, \\
S = \{1,4\},
\end{cases} \\
&\bS_{0,-2} V \otimes \bS_{3,1} U \otimes A(-4), && S = \{2,4\}, \\
&\bS_{-2,-2} V \otimes \bS_{3,3} U \otimes A(-6), && S = \{3,4\}.
\end{align*}
Grouping the terms according to the parity of the odd elements of \( S \) yields representations of the same dimension. For example, when \( U = V \), the summand
$\bS_{1,-1} V \otimes \bS_2 V$
appears in the decomposition. On the other hand, the top and bottom $\rho$-weights satisfy
$\rho_\lambda^{\mathrm{top}} = (1,-1)$ and  $\rho_\lambda^{\mathrm{bot}} = (2,0)$,
so the decomposition agrees with the prediction of Theorem~\ref{thm:A-rho}.
\end{example}

\begin{example}
Consider \(\dim V = 2\), \(\dim U = 2\), and the representation
\[
L_{(2,0,0,0)}^{\fgl (V \oplus U)} = \Sym^2(V \oplus U).
\]
Its free resolution over \( A = \Sym(V^* \otimes U) \) has the following terms, along with the corresponding subsets \( S \subset [4] \):
\begin{align*}
&\bS_{2} V \otimes A, && S = \{1,2\}, \\
&\bS_{2,-1} V \otimes U \otimes A(-1), && S = \{1,3\}, \\
&\bS_{2,-2} V \otimes \bS_{1,1} U \otimes A(-2) \
\oplus \  \bS_{-1,-1} V \otimes \bS_{4,0} U \otimes A(-4),
&& \begin{cases}
S = \{1,4\}, \\
S = \{2,3\},
\end{cases} \\
&\bS_{-1,-2} V \otimes \bS_{4,1} U \otimes A(-5), && S = \{2,4\}, \\
&\bS_{-2,-2} V \otimes \bS_{4,2} U \otimes A(-6), && S = \{3,4\}.
\end{align*}

Grouping the terms according to the parity of the odd entries of \( S \) yields representations of the same dimension. When \( U = V \), the summand $\bS_{2,-1} V \otimes \bS_1 V$ appears in the decomposition. The top and bottom $\rho$-weights are $\rho_\lambda^{\mathrm{top}} = (2,-1)$ and $\rho_\lambda^{\mathrm{bot}} = (1,0)$,
so the decomposition agrees with the prediction of Theorem~\ref{thm:A-rho}.
\end{example}

\begin{proof}[Proof of Theorem \ref{thm:A-rho}]
  \textbf{Proof of (4):} We use (3) and plug into the type A dimension formulas provided by \cite[Chapter 24]{fultonharris} to find:
\[
  \dim L^{\fgl_n}_{\rho^{\rm top}_\lambda} = \prod_{1 \leq i < j \leq n} \frac{\lambda_{2i-1} - \lambda_{2j-1} + 2(j-i)}{j-i}, \qquad
  \dim L^{\fgl_k}_{\rho^{\rm bot}_\lambda} = \prod_{1 \leq i < j \leq k} \frac{\lambda_{2i} - \lambda_{2j} + 2(j-i)}{j-i}.
  \]
  Pulling out a factor of $2$ from all terms in the above products gives a factor of $2^{\binom{n}{2} + \binom{k}{2}}$. If $k = n$ then this becomes $2^{n^2-n}$, which, combined with the extra factor of $2^n$ gives $2^{n^2} = 2^{\dim (V^* \otimes U)}$. If $k = n-1$ then the powers of $2$ instead combine to give $2^{n(n-1)} = 2^{\dim (V^* \otimes U)}$. 
\end{proof}

\subsection{Special case: Exterior Algebra}

Now we consider the special case of Theorem~\ref{thm:A-rho} with $\lambda = 0$ and $k=n$, where the combinatorics can be made more explicit. In this case, $\rH_i(U^* \otimes V; \bC) = \bigwedge^i(U^* \otimes V)$. Recall that the dual (or transposed) Cauchy identity gives the isomorphism
\[
  \bigwedge^\bullet(U^* \otimes V) = \bigoplus_{\lambda \subseteq n \times n} \bS_\lambda(U^*) \otimes \bS_{\lambda^T}(V).
\]
In the previous section, by taking $\lambda = 0$ and $V = U$, we get a partition of the set of $\lambda$ into $2^n$ blocks so that each sum of representations in each block is isomorphic to $\bS_\rho(V^*) \otimes \bS_\rho(V)$, where $\rho = \rho^\rA = (n-1, n-2, \dots , 0)$. 

We use the following bijection between $n$-element subsets of $[2n]$ and $\lambda \subseteq n \times n$: a partition $\lambda$ is determined by an integer walk from $(0,0)$ to $(n,n)$ that only goes up and to the right ($\lambda$ is the portion above the walk). Such a walk is determined by a subset by recording the positions that go up.

Given a subset $S$, let $\alpha(S)$ be the partition corresponding to the walk associated with $S$. Concretely, $\alpha(\{s_1,\dots,s_n\})$ is the partition $\lambda = (s_n-n, s_{n-1}-n+1 , \dots , s_1-1)$ if $s_1<\cdots<s_n$, and taking complements has the effect of taking the transpose of the complement of $\lambda$ in an $n \times n$ box (since horizontal steps become vertical steps in the transposed path). 

\begin{example}
  Let $n=4$. Consider the path of length \( 8 = 4+4 \) with up-steps at positions in \( S = \{1,2,6,8\} \) and right-steps elsewhere. The following diagram reveals that \(\alpha(S) = (4,3) \):
\[
  \text{Path: } U\ U\ R\ R\ R\ U\ R\ U \qquad 
\begin{tikzpicture}[scale=0.3]
  \draw[step=1,gray,very thin] (0,0) grid (4,4);
  
  \fill[orange!40] (0,2) rectangle (3,4);
  \fill[orange!40] (3,3) rectangle (4,4);
  
  \draw[thick] (0,0) -- (0,2) -- (3,2) -- (3,3) -- (4,3) -- (4,4);
\end{tikzpicture} \qedhere
\]
\end{example}

There is a simple relationship between $\alpha(S)$ and the weights appearing in $\bS_{\beta_0 (S)}$:

\begin{lemma}
  Let \( S \subseteq [2n] \) be a subset of size \( n \). We have:
\[
\bS_{\beta_0(S)}(V,V) = \bS_{\alpha(S)}(V^*) \otimes \bS_{\alpha(S)^T}(V).
\]
\end{lemma}

\begin{proof}
Write $S = \{ s_1 < \cdots < s_n \}$ and $[2n] \backslash S = \{ s_1' < \cdots < s_{n}' \}$. By definition of $\delta_{2n}= \delta = (2n-1 , 2n-2, \dots , 0)$ we find
\begin{align*}
  (\delta|_S  , \delta|_{S^c}) &= (2n-s_1 , 2n-s_2 , \dots , 2n-s_n ,  2n - s_1' , \dots , 2n-s_n' ),\\
  (\delta|_S  , \delta|_{S^c}) - \delta &= ( 1 - s_1, \dots , n- s_n , n+1-s_1', \dots , 2n - s_n' ).
\end{align*}
The tuple $( 1 - s_1, \dots , n- s_n )$ is the highest weight of $\bS_{\alpha(S)} (V^*)$ and $( n+1-s_1', \dots , 2n - s_n' )$ is the highest weight of $\bS_{\alpha(S)^T} (V)$, which yields the result. 
\end{proof}

\begin{example}\label{ex:typeAgrouping}
We give an explicit example of pairing up the factors appearing in the decomposition of $\bigwedge^\bullet (\fgl (V))$ when $\dim V = 3$. In Figure~\ref{fig:gl3}, we list:
\begin{itemize}
    \item All subsets $T \subseteq \{ 1,3,5 \}$,
    \item For a given $T$, all $S$ such that $S_{\rm odd}= T$,
    \item For a given $S$, the associated partition $\alpha(S)$ (written in red), and
    \item For a given partition $\lambda$, the weights showing up in the direct sum decomposition of the rational representation $\bS_\lambda (V^*) \otimes \bS_{\lambda^T} (V)$.
\end{itemize}
Just focusing on the rational weights, it is easy to verify directly that each block adds up to the same representation, and there are evidently $8 = 2^3$ total blocks.
\end{example}

\begin{figure}
\centering
{\ytableausetup{boxsize=0.5em}
\begin{tikzpicture}[>=latex]

  \node (top) at (0, 4.2) {\scriptsize 
  \begingroup\renewcommand{\arraystretch}{1.5}%
  \begin{tabular}{@{}l@{}}
      \textbf{$\varnothing$:} \\[0.7em]
      \shortstack[l]{\{2,4,6\}\\ \hspace{0.9em}{\color{red}\ydiagram{3,2,1}} $= [2;2] + [2;1,1] + [1,1;2] + 2 [1;1] + [1;1]$}
    \end{tabular}
  \endgroup
  };

  \node (left) at (-5.5, 1.0) {\scriptsize
  \begingroup\renewcommand{\arraystretch}{1.5}%
    \begin{tabular}{@{}l@{}}
      \textbf{\{1\}:} \\[0.7em]
      \shortstack[l]{\{\textbf{1},2,4\}\\ \hspace{0.9em}{\color{red}\ydiagram{1}} $= [1;1] + [0;0]$} \\[0.9em]
      \shortstack[l]{\{\textbf{1},2,6\}\\ \hspace{0.9em}{\color{red}\ydiagram{3}} $= [2;1,1]$} \\[0.9em]
      \shortstack[l]{\{\textbf{1},4,6\}\\ \hspace{0.9em}{\color{red}\ydiagram{3,2}} $= [2;2] + [1,1;2] + [1;1]$}
    \end{tabular}
  \endgroup
  };

  \node (middle) at (0, 1.0) {\scriptsize
  \begingroup\renewcommand{\arraystretch}{1.5}%
    \begin{tabular}{@{}l@{}}
      \textbf{\{3\}:} \\[0.7em]
      \shortstack[l]{\{2,\textbf{3},4\}\\ \hspace{0.9em}{\color{red}\ydiagram{1,1,1}} $= [1,1;2]$} \\[0.9em]
      \shortstack[l]{\{2,\textbf{3},6\}\\ \hspace{0.9em}{\color{red}\ydiagram{3,1,1}} $= [2;2] + [1;1] + [0;0]$} \\[0.9em]
      \shortstack[l]{\{\textbf{3},4,6\}\\ \hspace{0.9em}{\color{red}\ydiagram{3,2,2}} $= [2;1,1] + [1;1]$}
    \end{tabular}
  \endgroup
  };

  \node (right) at (5.5, 1.0) {\scriptsize
  \begingroup\renewcommand{\arraystretch}{1.5}%
    \begin{tabular}{@{}l@{}}
      \textbf{\{5\}:} \\[0.7em]
      \shortstack[l]{\{2,4,\textbf{5}\}\\ \hspace{0.9em}{\color{red}\ydiagram{2,2,1}} $= [2;2] + [2;1,1] + [1;1]$} \\[0.9em]
      \shortstack[l]{\{2,\textbf{5},6\}\\ \hspace{0.9em}{\color{red}\ydiagram{3,3,1}} $= [1,1;2] + [1;1]$} \\[0.9em]
      \shortstack[l]{\{4,\textbf{5},6\}\\ \hspace{0.9em}{\color{red}\ydiagram{3,3,3}} $= [0;0]$}
    \end{tabular}
  \endgroup
  };

  \node (leftb) at (-5.5,-3.8) {\scriptsize
  \begingroup\renewcommand{\arraystretch}{1.5}%
    \begin{tabular}{@{}l@{}}
      \textbf{\{1,3\}:} \\[0.7em]
      \shortstack[l]{\{\textbf{1},2,\textbf{3}\}\\ \hspace{0.9em}{\color{red} $\varnothing$} $= [0;0]$} \\[0.9em]
      \shortstack[l]{\{\textbf{1},\textbf{3},4\}\\ \hspace{0.9em}{\color{red}\ydiagram{1,1}} $= [1,1;2] + [1;1]$} \\[0.9em]
      \shortstack[l]{\{\textbf{1},\textbf{3},6\}\\ \hspace{0.9em}{\color{red}\ydiagram{3,1,1}} $= [2;2] + [2;1,1] + [1;1]$}
    \end{tabular}
  \endgroup
  };

  \node (middleb) at (0,-3.8) {\scriptsize
  \begingroup\renewcommand{\arraystretch}{1.5}%
    \begin{tabular}{@{}l@{}}
      \textbf{\{1,5\}:} \\[0.7em]
      \shortstack[l]{\{\textbf{1},2,\textbf{5}\}\\ \hspace{0.9em}{\color{red}\ydiagram{2}} $= [2;1,1] + [1;1]$} \\[0.9em]
      \shortstack[l]{\{\textbf{1},4,\textbf{5}\}\\ \hspace{0.9em}{\color{red}\ydiagram{2,2}} $= [2;2] + [1;1] + [0;0]$} \\[0.9em]
      \shortstack[l]{\{\textbf{1},\textbf{5},6\}\\ \hspace{0.9em}{\color{red}\ydiagram{3,3}} $= [1,1;2]$}
    \end{tabular}
  \endgroup
  };

  \node (rightb) at (5.5,-3.8) {\scriptsize
  \begingroup\renewcommand{\arraystretch}{1.5}%
    \begin{tabular}{@{}l@{}}
      \textbf{\{3,5\}:} \\[0.7em]
      \shortstack[l]{\{2,\textbf{3},\textbf{5}\}\\ \hspace{0.9em}{\color{red}\ydiagram{2,1,1}} $= [2;2] + [1,1;2] + [1;1]$} \\[0.9em]
      \shortstack[l]{\{\textbf{3},4,\textbf{5}\}\\ \hspace{0.9em}{\color{red}\ydiagram{2,2,2}} $= [2;1,1]$} \\[0.9em]
      \shortstack[l]{\{\textbf{3},\textbf{5},6\}\\ \hspace{0.9em}{\color{red}\ydiagram{3,3,2}} $= [1;1] + [0;0]$}
    \end{tabular}
  \endgroup
  };

  \node (middlec) at (0, -6.9) {\scriptsize
  \begingroup\renewcommand{\arraystretch}{1.5}%
    \begin{tabular}{@{}l@{}}
      \textbf{\{1,3,5\}:} \\[0.7em]
      \shortstack[l]{\{\textbf{1},\textbf{3},\textbf{5}\}\\ \hspace{0.9em}{\color{red}\ydiagram{2,1}} $= [2;2] + [2;1,1] + [1,1;2] + 2 [1;1] + [1;1]$}
    \end{tabular}
  \endgroup
  };

  \draw[-] (top) -- (left);
  \draw[-] (top) -- (middle);
  \draw[-] (top) -- (right);

  \draw[-] (left) -- (leftb);
  \draw[-] (left) -- (middleb);
  \draw[-] (middle) -- (leftb);
  \draw[-] (middle) -- (rightb);
  \draw[-] (right) -- (rightb);
  \draw[-] (right) -- (middleb);

  \draw[-] (middleb) -- (middlec);
  \draw[-] (leftb) -- (middlec);
  \draw[-] (rightb) -- (middlec);

\end{tikzpicture}
}
\caption{Decomposition for $\bigwedge^\bullet(\fgl_3)$.}
\label{fig:gl3}
\end{figure}

\section{Type BCD generalized $\rho$-decomposition}\label{sec:typeBCDrhodecomp}

\subsection{Setup}

Let $V$ be an $m$-dimensional complex vector space. As before, write $m = 2n+\eps$ with $\eps \in \{0,1\}$. We consider three cases:

\medskip\noindent
{\bf Type B:} Define $\bV = V^* \oplus \bC \oplus V$, and equip it with the orthogonal form
\[
  \langle (f, e, v), (f', e', v') \rangle = f(v') + ee' + f'(v).
\]
Let $\fg = \fso(\bV) \cong \fso_{2m+1}$ be the orthogonal Lie algebra. Consider the parabolic subalgebra induced by the decomposition
\[
  \fg = \bigwedge^2 (\bV) = \underbrace{\bigwedge^2 V^* \oplus V^*}_{\fn_-} \oplus \underbrace{\fgl (V)}_{\fl} \oplus \underbrace{V \oplus \bigwedge^2 V}_{\fn}.
\]
We take $\gamma = \tfrac12$ and let $G = \Spin(2m+1, \bC)$ denote the spin group (the simply-connected Lie group whose Lie algebra is $\fg$).

\medskip\noindent
{\bf Type C:} Define $\bV = V^* \oplus V$, and equip it with the symplectic form
\[
  \langle (f, v), (f',v') \rangle = f(v') - f'(v).
\]
Let $\fg = \fsp(\bV) \cong \fsp_{2m}$ be the symplectic Lie algebra. Consider the parabolic subalgebra induced by the decomposition
\[
  \fg = \Sym^2 (\bV) = \underbrace{\Sym^2 V^*}_{\fn_-} \oplus \underbrace{\fgl (V)}_{\fl} \oplus \underbrace{\Sym^2 V}_{\fn}.
\]
We take $\gamma = 0$ and let $G = \Sp(2m, \bC)$ denote the symplectic group (simply-connected).

\medskip\noindent
{\bf Type D:} Define $\bV = V^* \oplus V$, and equip it with the orthogonal form
\[
  \langle (f, v), (f',v') \rangle = f(v') + f'(v).
\]
Let $\fg = \fso(\bV) \cong \fso_{2m}$ be the orthogonal Lie algebra. Consider the parabolic subalgebra induced by the decomposition
\[
  \fg = \bigwedge^2 (\bV) = \underbrace{\bigwedge^2 V^*}_{\fn_-} \oplus \underbrace{\fgl (V)}_{\fl} \oplus \underbrace{\bigwedge^2 V}_{\fn}.
\]
We take $\gamma=1$ and let $G = \Pin(2m,\bC)$ denote the pin group, whose identity component is the simply-connected form of the Lie group of $\fg$.

\subsection{Weyl group}

Let $W = (\bZ/2)^m \rtimes \fS_m$ be the hyperoctahedral group, which acts on $\bC^m$ as signed permutation matrices. This is the Weyl group in types B and C, but a variation is needed in type D. First we will flesh out the relevant combinatorics for $W^\fp$. Given a set $S$, we let $\cP(S)$ denote its lattice of subsets.

  \begin{definition}[Extended Gale Order]\label{def:extendedGale}
Let \( A, B \subseteq [m] \), and write their elements in increasing order:
$A = \{a_1 < a_2 < \dots < a_k\}$ and $B = \{b_1 < b_2 < \dots < b_\ell\}$.
﻿
We define the \defi{extended Gale order} \( \leq_{\text{Gale}} \) on \( \mathcal{P}([m]) \) by:
\( A \leq_{\text{Gale}} B \) if there exist indices \( 1 \leq j_1 < j_2 < \dots < j_k \leq \ell \) such that 
$a_i \leq b_{j_i}$ for all  $i = 1, \dots, k$.
The rank of a given subset $S = \{ a_1 , \dots , a_\ell \} \subseteq [m]$ is 
\[
  \rank S = \Sigma(S) := \sum_{i} a_i. \qedhere
\]
\end{definition}
Here \(\rank S=\Sigma(S)\) matches the Coxeter length on \(W^\fp\) (see Proposition~\ref{prop:typeCgaleAndBruhat}); compare with the fixed-size Gale order in type~A.

\begin{figure}
  \[
    \begin{array}{c|c}
          \adjustbox{scale=.7}{
      \begin{tikzcd}
	& {\{1,2,3 \}} \\
	& {\{2,3\}} \\
	& {\{1,3 \}} \\
	{\{1,2\}} && {\{3\}} \\
	& {\{2\}} \\
	& {\{1\}} \\
	& \varnothing
	\arrow[no head, from=1-2, to=2-2]
	\arrow[no head, from=2-2, to=3-2]
	\arrow[no head, from=3-2, to=4-1]
	\arrow[no head, from=3-2, to=4-3]
	\arrow[no head, from=4-1, to=5-2]
	\arrow[no head, from=4-3, to=5-2]
	\arrow[no head, from=5-2, to=6-2]
	\arrow[no head, from=7-2, to=6-2]
      \end{tikzcd}}
        &
      \adjustbox{scale=.7}{\begin{tikzcd}
	& {\{1,2,3,4 \}} \\
	& {\{3,4\}} \\
	& {\{2,4\}} \\
	{\{2,3\}} && {\{1,4\}} \\
	& {\{1,3\}} \\
	& {\{1,2\}} \\
	& \varnothing
	\arrow[no head, from=1-2, to=2-2]
	\arrow[no head, from=2-2, to=3-2]
	\arrow[no head, from=3-2, to=4-1]
	\arrow[no head, from=3-2, to=4-3]
	\arrow[no head, from=4-1, to=5-2]
	\arrow[no head, from=4-3, to=5-2]
	\arrow[no head, from=5-2, to=6-2]
	\arrow[no head, from=7-2, to=6-2]
      \end{tikzcd}
\qquad
           \begin{tikzcd}
	& {\{2,3,4 \}} \\
	& {\{1,3,4\}} \\
	& {\{1,2,4\}} \\
	{\{1,2,3\}} && {\{4\}} \\
	& {\{3\}} \\
	& {\{2\}} \\
	& \{1\}
	\arrow[no head, from=1-2, to=2-2]
	\arrow[no head, from=2-2, to=3-2]
	\arrow[no head, from=3-2, to=4-1]
	\arrow[no head, from=3-2, to=4-3]
	\arrow[no head, from=4-1, to=5-2]
	\arrow[no head, from=4-3, to=5-2]
	\arrow[no head, from=5-2, to=6-2]
	\arrow[no head, from=7-2, to=6-2]
      \end{tikzcd}}
    \end{array}
  \]
  \caption{Left: Hasse diagram for extended Gale order on $\cP([3])$.\\ Right: Hasse diagram for semi-extended Gale order on $\cP([4])$.}
  \label{fig:hasse}
\end{figure}
Figure~\ref{fig:hasse} shows the Hasse diagram for the extended Gale order on $\mathcal{P} ([3])$.

\begin{proposition}\label{prop:typeCgaleAndBruhat}
  \begin{enumerate}
  \item Given $T \subseteq [m]$, define $w_T$ by $w_T^{-1} \bullet 0 = \beta_0(T^{\inv})$ (where $\beta_0 (-)$ is defined in \S\ref{sec:typeBCDidentities}). Then $T \mapsto w_T$ is an order-preserving bijection between $\cP([m])$ equipped with the extended Gale order and $W^\fp$.
\item For all $T$, we have $\ell(w_T) = \Sigma(T) = \sum_{t \in T} t$.
\end{enumerate}
\end{proposition}

\begin{proof}
    Assume that $S = \{ s_1 < \cdots < s_p \}$ is covered by $T = \{ t_1 < \cdots < t_q \}$ with respect to the extended Gale ordering. By definition, we either have that $T = \{ 1  < s_1 < s_2 < \cdots < s_p \}$, or there is a distinguished $1 \leq i \leq q$ such that $t_j = s_j$ for all $j \neq i$ and $t_i = s_i+1$. In the latter case, the fact that $\ell(w_T) = \ell(w_S)+1$ is identical to the argument used in the proof of Proposition \ref{prop:galeAndBruhat}. In the former case, $w_{T}$ is related to $w_{S}$ by multiplying $w_S$ on the right by the element that negates position $m$. By definition of the length function it follows that $\ell (w_T) = \ell (w_S) + 1$. The statement about ranks is now an immediate consequence.
\end{proof}

It turns out that another variant of the Gale order describes the poset structure induced by the Bruhat order for type D. First, we recall that $W(\rD_m)$ is the index 2 subgroup of $W(\rC_m)$ consisting of signed permutations with an even number of signs. We let $\sigma$ denote the signed permutation that negates the $m$th coordinate and is the identity elsewhere. This is a representative for the nontrivial coset. We define $\sigma W^\fp = \{ \sigma w \mid w \in W^\fp\}$. This inherits an ordering by $\sigma w \le \sigma w'$ if and only if $w \le w'$ (Bruhat order for $W(\rD_m)$). By abuse of notation, we will call this the Bruhat order on $W^\fp \cup \sigma W^\fp$, but we emphasize that this does not agree with the Bruhat order coming from $W(\rC_m)$.

  \begin{definition}[Semi-extended Gale Order]\label{def:semi-extendedGale}
Let \( A, B \subseteq [m] \), and write their elements in increasing order:
$A = \{a_1 < a_2 < \dots < a_k\}$ and $B = \{b_1 < b_2 < \dots < b_\ell\}$.

We define the \defi{semi-extended Gale order} \( \leq_{\text{sGale}} \) on \( \mathcal{P}([m]) \) by:
\( A \leq_{\text{sGale}} B \) if $k\equiv \ell \pmod 2$ and there are indices \( 1 \leq j_1 < j_2 < \dots < j_k \leq \ell \) such that 
$a_i \leq b_{j_i}$ for all  $i = 1, \dots, k$.
The rank of a given subset $S = \{ a_1 , \dots , a_\ell \} \subseteq [m]$ is given by
\[
  \rank S = \Sigma(S)-|S| = \sum_{i} (a_i-1). \qedhere
\]
\end{definition}

Figure~\ref{fig:hasse} shows the Hasse diagram for the semi-extended Gale order on $\mathcal{P} ([4])$. Note that there is an obvious order-preserving bijection between the two components: $T \mapsto T \cup \{1\}$ if $1 \notin T$ and $T \mapsto T \setminus \{1\}$ otherwise.

  Let $\cP([m])$ denote the set of subsets of $[m]$, endowed with the semi-extended Gale order. We let $\cP_{\rm even}([m])$, respectively $\cP_{\rm odd}([m])$, denote the collection of even, respectively odd, sized subsets.

\begin{proposition}
  \begin{enumerate}
  \item Given $T \subseteq [m]$, define $w_T$ by $w_T^{-1} \bullet 0 = \beta_0(T^{\inv})$. Then $T \mapsto w_T$ is an order-preserving bijection between $\cP([m])$ equipped with the semi-extended Gale order and $W^\fp \cup \sigma W^\fp$. Under this bijection, $\cP_{\rm even}([m])$ is identified with $W^\fp$.
\item For all $T$, we have $\ell(w_T) = \Sigma(T) - |T| = \sum_{t \in T} (t-1)$.
\end{enumerate}
\end{proposition}

In the notation of \S\ref{sec:typeBCDidentities}, notice that Kostant's theorem in all of the above cases implies that there is an equality
$$\bH_\lambda (z_1 , \dots , z_m; t) =  \sum_{i \geq 0} \ch \rH_i (\fn_- ; L_\lambda) \cdot t^i,$$
where we recall the convention that
\[
(z_1,\dots,z_m)=
\begin{cases}
(y_1,\dots,y_n,\,y_1^{-1},\dots,y_n^{-1}) & \text{if $m=2n$},\\
(y_1,\dots,y_n,\,y_1^{-1},\dots,y_n^{-1},\,1) & \text{if $m=2n+1$}.
\end{cases}
\]
For sake of completeness, let us briefly explain how Kostant's theorem gives the above formula. In type D, we use the parametrization by \(\cP([m])\) coming from
\(W^\fp\cup \sigma W^\fp\) as in the preceding proposition. Thus the summands are indexed by subsets \(T\subseteq [m]\), with corresponding signed permutations \(w_T\).

By construction, the \(w_T\)-summand has character 
\[
  \ch L^{\fgl(V)}_{w_T^{-1}\bullet \lambda}
  =
  \ch L^{\fgl(V)}_{\beta_\lambda(T^\inv)}
=
s_{\beta_\lambda(T^{\inv})}(z_1,\dots,z_m).
\]
Moreover, by the preceding propositions,
\[
\ell(w_T)=
\begin{cases}
\Sigma(T) & \text{in types B and C},\\
\Sigma(T)-|T| & \text{in type D},
\end{cases}
\]
which is exactly \(\rank(T)\). Therefore
\[
\sum_{i\ge 0}\ch \rH_i(\fn_-;L_\lambda)\,t^i
=
\sum_{T\subseteq [m]}
s_{\beta_\lambda(T^{\inv})}(z_1,\dots,z_m)\,t^{\rank(T)}.
\]
and we get our claim upon relabeling \(S=T^{\inv}\).

\subsection{Theorems}

Let $V$ be an $m$-dimensional vector space with $m = 2n+\eps$, where $\eps \in \{ 0 ,1 \}$ as before. For the sake of clarity, recall that the previous section established that $G$ denotes the following group, depending on which setting we are in:

\medskip

\textbf{Type B:} $G = \Spin (2m+1, \bC)$ and $\gamma = \tfrac12$,

\textbf{Type C:} $G = \Sp (2m, \bC)$ and $\gamma = 0$,

    \textbf{Type D:} $G = \Pin(2m, \bC)$ and $\gamma = 1$.

    \medskip

    Recall as well that we use the notation $\gamma' = \lfloor \gamma \rfloor$. Let $L_\lambda$ denote the irreducible $G$-representation with highest weight $\lambda$. The group $G$ has a maximal torus of rank $m$, and in the statements below we consider characters of an $n$-dimensional subtorus with parameters $y_1,\dots,y_n$ (and set $y_{n+1}=1$). Note that here our convention for the torus parameters matches the usage of $y_1, \dots , y_n$ and $z_1 , \dots , z_m$ from Section \ref{sec:typeBCDidentities}.

Define $\zeta(\lambda) = \begin{cases} 1 & \text{if $\lambda_m = 0$} \\ 0 & \text{else} \end{cases}$; this will only be relevant in the case of type D.

\begin{theorem} \phantomsection \label{thm:kostantBCD} 
  \begin{enumerate} 
  \item  The graded character $\sum_{i \ge 0} {\rm char}\ \rH_i(\fn_-; L_\lambda) \cdot t^i$ is divisible by $(1+t)^n$.

\item   For any two subsets $T, T' \subseteq [m]_{1-\eps}$, we have
  \[
    \sum_{\substack{S \subseteq [m] \\ S_{1-\eps} = T}} {\rm char}\ \bS_{\beta_\lambda(S)}(V) =     \sum_{\substack{S \subseteq [m] \\ S_{1-\eps} = T'}} {\rm char}\ \bS_{\beta_\lambda(S)}(V).
  \]

\item There is an equality:
  \begin{align*}
    \ch \rH_\bullet (\fn_-; L_\lambda) &= 2^{n-\gamma'\zeta(\lambda)} \sum_{\substack{S \subseteq [m]_\eps}} {\rm char}\ \bS_{\beta_\lambda(S)}(V)\\
&=    2^{n - \gamma'\zeta(\lambda)} s_{\rho^{\rm top}_\lambda}^\rC (y_1, \dots , y_{n}) s^\rD_{\rho^{\rm bot}_\lambda} (y_1 , \dots , y_{n + \eps}),
  \end{align*}
  where
\begin{align*}
    \rho^{\rm top}_\lambda &:= (\lambda_{1 + \eps} + n - \gamma, \lambda_{3 + \eps} + n -1 - \gamma, \dots , \lambda_{2n-1+\eps} + 1 - \gamma)\\
    \rho^{\rm bot}_\lambda &:= (\lambda_{2- \eps} + n + \eps - \gamma, \lambda_{4- \eps} + n + \eps - 1-\gamma, \dots, \lambda_{2n + \eps} +1 - \gamma).
\end{align*}

\item The total dimension satisfies
  \begin{align*}
    \dim \rH_\bullet ( \fn_-; L_\lambda)
     &= 2^{n+1 - \gamma'\zeta(\lambda)} \cdot \prod_{\substack{1 \leq i < j \leq m\\ i \equiv j \pmod 2}} 2\left(1 + \frac{\lambda_{i} - \lambda_{j}}{j-i} \right) \cdot \Xi_1 \cdot \Xi_2
  \end{align*}
  where
  \[
    \Xi_1=\prod_{1 \leq i \leq j \leq n} \left(2+  \frac{\lambda_{2i-1+\eps} + \lambda_{2j-1+\eps} - 2\gamma}{2n+2 - i - j} \right),
  \quad    \Xi_2 = \prod_{1 \leq i < j \leq n+\eps} \left( 2 + \frac{2 + \lambda_{2i-\eps} + \lambda_{2j-\eps} -2\gamma}{2n + 2\eps - i - j} \right).
\]
\end{enumerate}
\end{theorem}

This formula can be simplified quite a bit, but this simplification depends on $\gamma$. Hence, we will handle each case separately in \S\S\ref{sec:typeCex}, \ref{sec:typeDex}, and \ref{sec:typeBex}.

\begin{proof}
  \textbf{Proof of (4):}   We employ \cite[Formulas 24.19, 24.41]{fultonharris} to find:
    \begin{align*}
      s^\rC_{\rho^{\rm top}_\lambda}(1,\dots,1) &= \prod_{1 \leq i < j \leq n} \frac{\lambda_{2i-1+\eps} - \lambda_{2j-1 + \eps} + 2(j-i)}{j-i} \\
      & \qquad \cdot \prod_{1 \leq i \leq j \leq n} \frac{4n+4 + \lambda_{2i-1+\eps} + \lambda_{2j-1+\eps} - 2(i + j + \gamma)}{2n+2 - i - j},\\
      s^\rD_{\rho^{\rm bot}_\lambda}(1,\dots,1) &= 2 \cdot \prod_{1 \leq i < j \leq n+ \eps} \frac{\lambda_{2i - \eps} - \lambda_{2j - \eps} + 2(j-i)}{j-i} \\
      & \qquad \cdot \prod_{1 \leq i < j \leq n+\eps} \frac{4n+4\eps +2 + \lambda_{2i-\eps} + \lambda_{2j-\eps} - 2(i + j + \gamma)}{2n + 2\eps - i - j}. 
    \end{align*}

    We can combine the first product in each to get
    \[
      \prod_{\substack{1 \leq i < j \leq m\\ i \equiv j \pmod 2}} 2\left(1 + \frac{\lambda_{i} - \lambda_{j}}{j-i} \right).
    \]
    The remaining two products are evidently $\Xi_1$ and $\Xi_2$.
  \end{proof}

  \section{Type C Examples}\label{sec:typeCex}

\subsection{Dimension Formula}
   
  \begin{proposition}
  Assume the Type~C setup of \S\ref{sec:typeBCDrhodecomp}. Then there is an equality:
    \begin{align*}
      \dim \rH_\bullet ( \fn_-; L_\lambda) &= 2^{\dim \fn_-} \prod_{\substack{1 \leq i < j \leq m\\ i \equiv j \pmod 2}} \left(1 + \frac{\lambda_{i} - \lambda_{j}}{j-i} \right) \left( 1 + \frac{\lambda_{i} + \lambda_{j} }{2(m + 1)-i-j} \right)\\
      &\qquad \cdot \prod_{1 \leq i \le n} \left(1 + \frac{\lambda_{2i-1+\eps} }{2(n+1 - i)} \right).
    \end{align*}
  \end{proposition}

  \begin{remark}
      Observe that the statement of Corollary \ref{cor:2n+2n-1 conj} in the type C setting is immediate from the above dimension formula: there are two cases. If not all terms of $\lambda$ are equal, then the argument is identical to that used in Remark \ref{rk:2nCorJustification}. If all $\lambda$ are equal and positive (i.e., $L_\lambda$ is not the trivial representation), then, in particular $\lambda_{2n-1+\eps} \geq 1$. The $i = n$ factor of the product $\prod_{1 \leq i \le n} \left(1 + \frac{\lambda_{2i-1+\eps} }{2(n+1 - i)} \right)$ is then $1 + \lambda_{2n-1+\eps}/2 \geq 3/2$. This yields the statement.
  \end{remark}

  \begin{proof}
We will simplify $\Xi_2$ from Theorem~\ref{thm:kostantBCD}, which is given by
\[
  \Xi_2 = \prod_{1 \leq i < j \leq n+\eps}  \frac{2(2n+2\eps-i-j+1) + \lambda_{2i-\eps} + \lambda_{2j-\eps}}{2n + 2\eps - i - j},
\]
by shuffling the order of the numerators and denominators. First, we separate the product of the numerators by isolating the cases when $j=i+1$:
    \[
      \prod_{\substack{1 \le i <j \le n+\eps\\ j = i+1}} 4 \left( n+\eps-i + \frac{\lambda_{2i-\eps} + \lambda_{2i+2-\eps}}{4} \right) \cdot \prod_{\substack{1 \le i < j \le n+\eps\\ j \ne i+1}}2 \left( 2n+2\eps+1 -i-j+ \frac{\lambda_{2i-\eps} + \lambda_{2j-\eps}}{2} \right).
    \]
    Second, we separate the product of the denominators by isolating the cases when $j=n+\eps$:
    \[
      \prod_{1 \le i \le n+\eps - 1} ( n + \eps - i) \cdot \prod_{\substack{1 \le i
          < j \le n+\eps-1}} (2n+2\eps - i - j).
    \]
    We have indexed the second product so that if we do the substitution $j \mapsto j+1$, it matches the previous second product. Finally, we combine the terms:
    \begin{align*}
      \Xi_2 &= \prod_{\substack{1 \le i <j \le n+\eps\\ j = i+1}} 4 \left( 1 + \frac{\lambda_{2i-\eps}+ \lambda_{2i+2-\eps}}{4 (n+\eps-i)} \right) \cdot \prod_{\substack{1 \le i<j \le n+\eps\\ j \ne i+1}}2 \left( 1 + \frac{\lambda_{2i-\eps} + \lambda_{2j-\eps}}{2(2n+2\eps + 1 -i-j)} \right)\\
      &= 2^{n+\eps-1} \prod_{1 \le i < j \le n+\eps} 2 \left( 1 + \frac{\lambda_{2i-\eps} + \lambda_{2j-\eps}}{2(2n+2\eps + 1 -i-j)} \right).
    \end{align*}
    
    Now we combine everything; we have written the terms to suggest pulling out powers of 2. In total, the power of 2 that we get is
    \[
      n+1 + \binom{n}{2} + \binom{n+1}{2} + \binom{n+\eps}{2} + (n+\eps-1) + \binom{n+\eps}{2} = \dim \Sym^2 V. \qedhere
    \]
  \end{proof}

\begin{example}
Consider the Lie algebra \(\mathfrak{sp}(V \oplus V^*)\) with \(\dim V = 4\). Then $A = \Sym(\Sym^2 V)$ and consider the representation 
\[
L_{(1,0,0,0)}^{\fsp (V \oplus V^*)} = V^* \oplus V.
\]
The free resolution of this \( A \)-module has the following terms and subsets \( S \subset [4] \):
\begin{align*}
&V^* \otimes A, && S = \emptyset, \\
&\bS_{2,0,0,-1} V \otimes A(-1), && S = \{4\}, \\
&\bS_{3,1,0,-1} V \otimes A(-2), && S = \{3\}, \\
&\bS_{4,1,1,-1} V \otimes A(-3) \ \oplus \ \bS_{3,3,0,-1} V \otimes A(-3), && S = \{2\}, \{3,4\}, \\
&\bS_{6,1,1,1} V \otimes A(-5) \ \oplus \ \bS_{4,3,1,-1} V \otimes A(-4), && S = \{1\}, \{2,4\}, \\
&\bS_{6,3,1,1} V \otimes A(-6) \ \oplus \ \bS_{4,4,2,-1} V \otimes A(-5), && S = \{1,4\}, \{2,3\}, \\
&\bS_{6,4,2,1} V \otimes A(-7) \ \oplus \ \bS_{4,4,4,-1} V \otimes A(-6), && S = \{1,3\}, \{2,3,4\}, \\
&\bS_{6,5,2,2} V \otimes A(-8) \ \oplus \ \bS_{6,4,4,1} V \otimes A(-8), && S = \{1,2\}, \{1,3,4\}, \\
&\bS_{6,5,4,2} V \otimes A(-9), && S = \{1,2,4\}, \\
&\bS_{6,5,5,3} V \otimes A(-10), && S = \{1,2,3\}, \\
&\bS_{6,5,5,5} V \otimes A(-11), && S = \{1,2,3,4\}.
\end{align*}
The grouped parts for \( T = \emptyset, \{1\}, \{3\}, \{1,3\} \) have total dimensions:
\[
560 = 4 + 36 + 160 + 360 = 56 + 224 + 140 + 140.
\]
We also have $\rho_\lambda^{\mathrm{top}} = (3,1)$ and $\rho_\lambda^{\mathrm{bot}} = (2,1)$. Each decomposition under restriction gives the representation
\[
L_{(3,1)}^{\fsp_4}  \otimes L_{(2,1)}^{\fso_4},
\]
where the dimensions are $\dim L_{(3,1)}^{\fsp_4}  = 35$ and $\dim L_{(2,1)}^{\fso_4}  = 16$. 
\end{example}

\begin{example}
    Let $m=2n+1 = 3$, $\lambda = (1)$. The highest weights of the representations appearing in the minimal free resolution we get are:
    \[
\begin{aligned}
&\{1,\ 0,\ 0\}, && \{-1,\ -1,\ -5\}, && \{1,\ -1,\ -3\}, && \{-2,\ -4,\ -5\}, \\
&\{1,\ 0,\ -2\}, && \{-1,\ -3,\ -5\}, && \{1,\ -3,\ -3\}, && \{-4,\ -4,\ -5\}
\end{aligned}
\]
Taking the dimensions of all of these yields the list:
\[
\{\, 3,\ 15,\ 27,\ 15,\ 15,\ 27,\ 15,\ 3\}.
\]
Adding all of these up we obtain a total dimension of $120$. On the other hand, we compute:
$\rho^{\rm top}_{(1)} = (1)$ and $\rho^{\rm bot}_{(1)} = (3,1)$. The corresponding dimensions are given by:
$$\dim L_{(1)}^{\fsp_1} = 2, \quad \dim L_{(3,1)}^{\Pin_2} = 30.$$
Our formula implies that the total dimension should be equal to  $2^1 \cdot 2 \cdot 30 = 120$.
\end{example}

\subsection{Exterior algebra}

For each $i$, let $h(i) = (i+1,1^{i-1})$ be the hook partition.

\begin{lemma}\label{lem:typeCHookLemma}
  For a subset $S \subseteq [m]$, we have $-\beta_0(S)^{\rm op} = \bigcup_{i \in S} h(m+1-i)$, where by union we mean to nest the corresponding hook shapes. Here, $\beta_0$ and the reverse–negation operation $-\!(\cdot )^{\rm op}$ are as in \S\ref{sec:typeBCDidentities}. 
\end{lemma}

\begin{proof}
  Let \( S = \{s_1 < \cdots < s_k\} \subseteq [m] \), and let \( S^c = \{s_1' < \cdots < s_{m-k}'\} \) denote its complement in \( [m] \). We have
  \begin{align*}
    -\beta_0(S)^{\rm op} &= (m+2-s_1, \dots, m+k+1-s_k, s'_{m-k}-(m-k), \dots, s_1' - 1).
  \end{align*}
  Now we proceed by induction: when $k = 0$ the result is clear.

  Now suppose $k \geq 1$. Then $s'_i=i$ for $i < s_1$, so that $-\beta_0(S)^{\rm op}$ has $m-s_1+1$ positive entries. Hence, if we delete the first entry and subtract 1 from the remaining $m-s_1$ positive entries, we are removing the hook $h(m+1-s_1)$ from the partition $-\beta_0(S)^{\rm op}$. However, it is also clear that the resulting sequence is $-\beta_0(\{s_2,\dots,s_k\})^{\rm op}$ where $m$ is now $m-1$.
\end{proof}

Now we consider the special case of Theorem~\ref{thm:kostantBCD} with $\lambda = 0$. We will now use Lemma~\ref{lem:typeCHookLemma} to convert subsets into unions of hook shapes.

This conversion exhibits one obvious symmetry: $-\beta_0(S)^{\rm op}$ and $-\beta_0(S^c)^{\rm op}$ are complementary partitions inside the $m \times (m+1)$ rectangle. In particular, the Schur functors are dual to one another up to a power of the determinant. We will restrict the action to the symplectic Lie algebra, in which case the two Schur functors are isomorphic.

\begin{example}\label{ex:typeCgrouping}
  We give an example of the groupings used in the proof of Theorem \ref{thm:kostantBCD}.

The proof of Theorem \ref{thm:kostantBCD} groups these terms as in Figure~\ref{fig:sp4}, where we also give the $\fsp_4$-weights of these representations and the bijection between subsets of $\{1,\dots,4\}$ and Schur functors appearing in $\bigwedge^\bullet(\Sym^2)$.
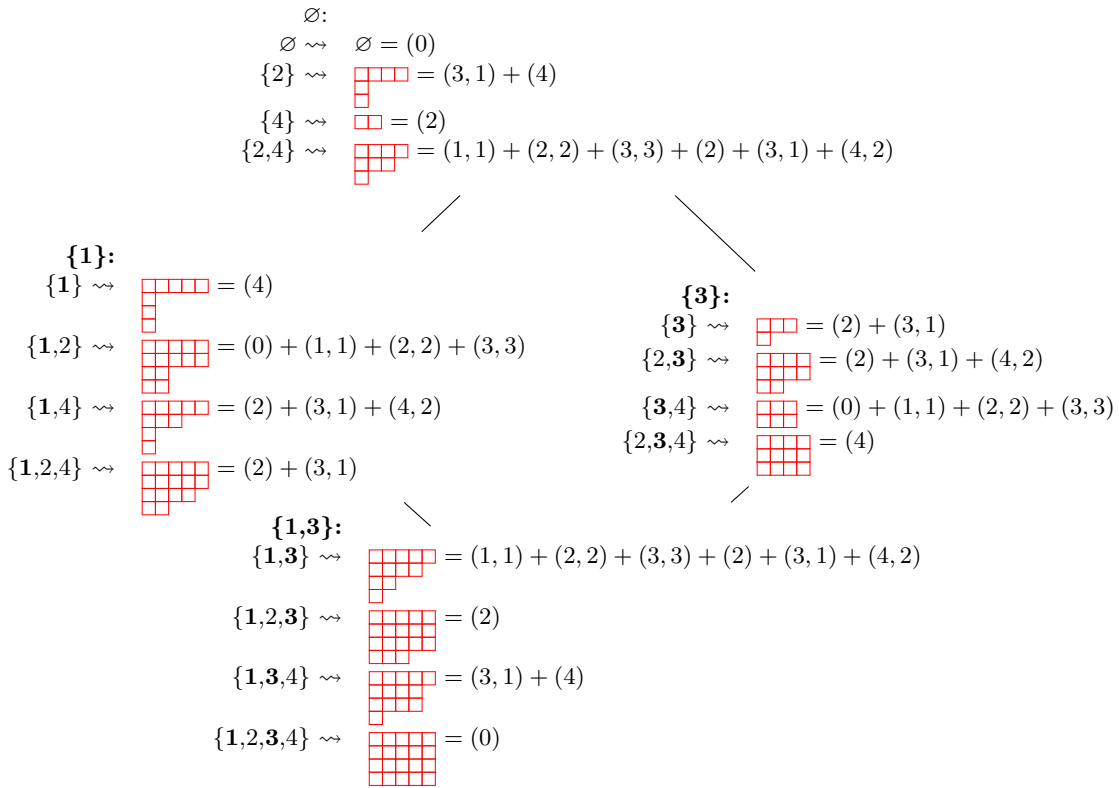
\begin{figure}
  {\ytableausetup{boxsize = 0.5em}
\begin{center}
\begin{tikzpicture}[>=latex]

  \node (top) at (0, 5.4) {\scriptsize 
  \begin{tabular}{rl}
      $\varnothing$: & \\
      $\varnothing \leadsto $ & $\varnothing = (0)$ \\
      \{2\} $\leadsto$ & {\color{red}\ydiagram{4,1,1}} $= (3,1) + (4)$ \\
      \{4\} $\leadsto$ & {\color{red}\ydiagram{2}}  $=(2)$ \\
      \{2,4\} $\leadsto$ & {\color{red}\ydiagram{4,3,1}} $=(1,1) + (2,2) + (3,3) + (2) + (3,1) + (4,2)$ 
    \end{tabular}
  };

  \node (left) at (-4, 1.6) {\scriptsize
    \begin{tabular}{rl}
      \textbf{\{1\}:} & \\
      \{\textbf{1}\} $\leadsto$ & {\color{red}\ydiagram{5,1,1,1}} $=(4)$  \\
      \{\textbf{1},2\} $\leadsto$ & {\color{red}\ydiagram{5,5,2,2}} $=(0) + (1,1) + (2,2) + (3,3)$ \\
      \{\textbf{1},4\} $\leadsto$ & {\color{red}\ydiagram{5,3,1,1}} $=(2) + (3,1) + (4,2)$\\
      \{\textbf{1},2,4\} $\leadsto$ & {\color{red}\ydiagram{5,5,4,2}} $=(2) + (3,1)$ 
    \end{tabular}
  };

  \node (right) at (4, 1.6) {\scriptsize
    \begin{tabular}{rl}
      \textbf{\{3\}:} & \\
      \{\textbf{3}\} $\leadsto$ & {\color{red}\ydiagram{3,1}} $=(2) + (3,1)$\\
      \{2,\textbf{3}\} $\leadsto$ & {\color{red}\ydiagram{4,4,2}} $=(2) + (3,1) + (4,2)$ \\
      \{\textbf{3},4\} $\leadsto$ & {\color{red}\ydiagram{3,3}} $= (0) + (1,1) + (2,2) + (3,3)$\\
      \{2,\textbf{3},4\} $\leadsto$ & {\color{red}\ydiagram{4,4,4}} $=(4)$
    \end{tabular}
  };

  \node (bottom) at (0, -2.0) {\scriptsize
    \begin{tabular}{rl}
      \textbf{\{1,3\}:} & \\
      \{\textbf{1},\textbf{3}\} $\leadsto$ & {\color{red}\ydiagram{5,4,2,1}} $=(1,1) + (2,2) + (3,3) + (2) + (3,1) + (4,2)$ \\
      \{\textbf{1},2,\textbf{3}\} $\leadsto$ & {\color{red}\ydiagram{5,5,5,3}} $=(2)$ \\
      \{\textbf{1},\textbf{3},4\} $\leadsto$ & {\color{red}\ydiagram{5,4,4,1}} $=(3,1) + (4)$\\
      \{\textbf{1},2,\textbf{3},4\} $\leadsto$ & {\color{red}\ydiagram{5,5,5,5}} $=(0)$ 
    \end{tabular}
  };

  \draw[-] (top) -- (left);
  \draw[-] (top) -- (right);

  \draw[-] (left) -- (bottom);
  \draw[-] (right) -- (bottom);

\end{tikzpicture}
\end{center}}
\caption{Decomposition for $\bigwedge^\bullet(\fsp_4)$.}
\label{fig:sp4}
\end{figure}
We can verify by inspection that each block is a direct sum of the same 10 (counting multiplicity) irreducible representations of $\fsp_4$.
\end{example}

\begin{remark}
One difference between the type A and C cases is that the size of each piece of the partitioning is constant in the latter case (notice that there are precisely $4$ partitions showing up in each piece of the above). Compare this to Example \ref{ex:typeAgrouping}.
\end{remark}

\begin{remark}
  When $\lambda=0$, we have $\rho^{\rm bot}_{(0)} = \rho^{\rm top}_{(0)} = (n,n-1,\dots,1)$. Furthermore, both have the same character as the symplectic Schur functor $\bS_{[n,n-1,\dots,1]}(\bC^{2n})$ by Proposition~\ref{prop:stair-sosp}.  This recovers Kostant's $\rho$-decomposition for $\fsp_{2n}$.
  \end{remark}

\section{Type D Examples}\label{sec:typeDex}

\subsection{Dimension Formula}

  \begin{proposition}
  Assume the Type~D setup of \S\ref{sec:typeBCDrhodecomp}. Then there is an equality:
    \begin{align*}
     \dim \rH_\bullet ( \fn_-; L_\lambda) &= 2^{\dim \fn_- + 1-\zeta(\lambda)} 
    \!\!\!\!\!\prod_{\substack{1 \leq i < j \leq m\\ i \equiv j \pmod 2}} \left(1 + \frac{\lambda_{i} - \lambda_{j}}{j-i} \right) \!\! \left( 1 + \frac{\lambda_{i} + \lambda_{j} }{2m-i-j} \right) \\ 
    &\qquad\qquad\qquad\cdot \prod_{1 \leq i \le n} \left(1 + \frac{\lambda_{2i-1+\eps} }{m+1-\eps - 2i} \right).
    \end{align*}
  \end{proposition}

\begin{remark}
  The power of 2 is $\dim \fn_-$ if $\lambda_m=0$ and is $\dim \fn_- + 1$ when $\lambda_m>0$. In the latter case, $L_\lambda$ is a direct sum of two finite length $\fn_-$-modules. However, the total ranks of the Tor groups of these two modules do not agree in general (see Example~\ref{ex:D-spin}).

Also, the proof of Corollary \ref{cor:2n+2n-1 conj} in this setting is again an immediate consequence of the above dimension formula. The case where not all parts of $\lambda$ are equal is again identical to the argument used in Remark \ref{rk:2nCorJustification}, and if all $\lambda_i$ are equal and positive, the $i=n$ term in the last product is $\ge 2$.
\end{remark}  

\begin{proof}
We will simplify $\Xi_1$ from Theorem~\ref{thm:kostantBCD}, which is given by
    \[
      \Xi_1=\prod_{1 \leq i \leq j \leq n} \frac{2 ( 2n+1-i-j) + \lambda_{2i-1+\eps} + \lambda_{2j-1+\eps}}{2n+2 - i - j} ,
      \]
      by shuffling the order of the numerators and denominators. First, we separate the product of the numerators by isolating the cases when $j=n$:
    \[
      \prod_{\substack{1 \leq i \leq n}} (2n+2 -2i + \lambda_{2i-1+\eps} + \lambda_{2n-1+\eps})
      \prod_{\substack{1 \leq i \leq j \leq n\\ j \ne n}} (4n+2 -2i-2j+ \lambda_{2i-1+\eps} + \lambda_{2j-1+\eps}).
    \]
    Second, we separate the product of the denominators by isolating the cases when $i=j$:
    \[
      \prod_{\substack{1 \leq i \leq n}} (2n+2 - 2i)
      \prod_{\substack{1 \leq i \leq j \leq n\\ i \ne j}} (2n+2 - i - j).
    \]
    We have indexed the second product so that if we do the substitution $j \mapsto j-1$, it matches the previous second product. Finally, we combine the terms:
    \begin{align*}
      \Xi_1 &= \prod_{\substack{1 \le i \le n}} \left( 1 + \frac{\lambda_{2i-1+\eps}+ \lambda_{2n-1+\eps}}{2n+2-2i} \right) \cdot \prod_{\substack{1 \le i\le j \le n-1}}2 \left( 1 + \frac{\lambda_{2i-1+\eps} + \lambda_{2j-1+\eps}}{2(2n+1 -i-j)} \right)\\
      &= 2^{-n} \prod_{1 \le i \le j \le n} 2 \left( 1 + \frac{\lambda_{2i-1+\eps} + \lambda_{2j-1+\eps}}{2(2n+1-i-j)} \right).
    \end{align*}
    Now we combine everything; we have written the terms to suggest pulling out powers of 2. In total, the power of 2 that we get is
    \[
      n + 1 - \zeta(\lambda) + \binom{n}{2} +2 \binom{n+\eps}{2} -n + \binom{n+1}{2} = \dim \bigwedge^2 V + 1 - \zeta(\lambda). \qedhere
    \]
  \end{proof}

\begin{example}
Consider \( A = \Sym(\bigwedge^2 V) \) with \( \dim V = 4 \). Take the representation
\[
  L_{(1,1,0,0)}^{\fso (V \oplus V^*)} = \bigwedge^2 (V^* \oplus V).
\]
The terms in the minimal free resolution of this \( A \)-module are:
\begin{align*}
&\bigwedge^2 V^* \otimes A, && S = \emptyset, \\
&\bS_{1,1,-1,-1} V \otimes A(-1), && S = \{1\}, \\
&\bS_{3,1,1,-1} V \otimes A(-3), && S = \{2\}, \\
&\bS_{3,2,2,-1} V \otimes A(-4) \oplus \bS_{4,1,1,0} V \otimes A(-4), && S = \{1,2\},\ \{3\}, \\
&\bS_{4,2,2,0} V \otimes A(-5), && S = \{1,3\}, \\
&\bS_{4,4,2,2} V \otimes A(-7), && S = \{2,3\}, \\
&\bS_{4,4,3,3} V \otimes A(-8), && S = \{1,2,3\}.
\end{align*}
We have $\rho_\lambda^{\mathrm{top}} = \rho_\lambda^{\mathrm{bot}} = (2,0)$. Each piece corresponding to a fixed \( T \subset \{1,3\} \) has total dimension $90$. Its character is equal to the character of the tensor product:
\[
L_{(2,0)}^{\fsp_4} \otimes L_{(2,0)}^{\fso_4},
\]
where $\dim L_{(2,0)}^{\fsp_4}  = 10$ and $\dim L_{(2,0)}^{\fso_4}  = 9$.
\end{example}

\begin{example} \label{ex:D-spin}
Consider arbitrary \( n \), and let \( A = \Sym(\bigwedge^2 V) \) with \( \dim V = 2n \). Take the half-spinor representation
\[
L_{\omega_{4n-1}}^{\fso_{4n}} = \bigwedge^{\mathrm{even}} V,
\]
the sum of all even-degree exterior powers of \( V \) (where $\omega_{4n-1} = (\frac12, \frac12, \dots , \frac12)$). Up to a factor of $1/2$ times the trace representation, the terms in the free resolution of this \( A \)-module are of the form
\[
\bS_\lambda V \otimes A(-|\lambda|/2),
\]
where \( \lambda \) is a self-conjugate partition with even rank (i.e., the number of squares along the main diagonal of the Young diagram of \( \lambda \) is even). We can compute the weights $\rho^{\rm top}_{\omega_{4n-1}}$ and $\rho^{\rm bot}_{\omega_{4n-1}}$.

Our methods allow us to compute only the total rank of the $\Pin$-representation with highest weight $(\frac12, \frac12, \dots , \frac12)$. Concretely, up to a factor of a trace representation this is simply the exterior algebra $\bigwedge^\bullet V$ viewed as a finite dimensional $\Sym (\bigwedge^2 V)$-module. The highest weights of the representations that appear in this resolution are precisely all self-conjugate partitions contained in a $2n \times 2n$ box. 

For example, assume $m = 2n =4$. After factoring out by $1/2$ times a trace representation, the representations appearing in the minimal free resolution of the even half-spinor representation are given by:
\[ 
      0 \to   \bS_{(4,4,4,4)} V \otimes A(-8)  \to \bS_{(4,4,2,2)} V \otimes A(-6)   \to    \bS_{(4,3,2,1)} V \otimes A(-5)  \] \[ \to \begin{matrix} \bS_{(3,3,2)} V \otimes A(-4) \\ \oplus \\ \bS_{(4,2,1,1)} V \otimes A(-4) \end{matrix} \to  \bS_{(3,2,1)} V \otimes A(-3) \to  \bS_{(2,2)} V \otimes A(-2) \to    A.
\]
Taking the ranks of all of these yields the list:
\[
\{\, 1,\ 20,\ 64,\ 90,\ 64,\ 20,\ 1 \,\}
\]
which yields a total rank of $260$. Likewise, for the odd half-spinor representation we obtain:
\[ 
      0 \to   \bS_{(4,4,4,3)} V \otimes A(-8)  \to \bS_{(4,4,3,2)} V \otimes A(-7)   \to    \bS_{(4,3,3,1)} V \otimes A(-6)  \] \[ \to \begin{matrix} \bS_{(3,3,3)} V \otimes A(-5) \\ \oplus \\ \bS_{(4,1,1,1)} V \otimes A(-4) \end{matrix} \to  \bS_{(3,1,1)} V \otimes A(-3) \to  \bS_{(2,1)} V \otimes A(-2) \to   V \otimes  A(-1).
\]
Taking ranks yields
\[
\{\, 4,\ 20,\ 36,\ 40,\ 36,\ 20,\ 4 \,\}
\]
adding up to a total rank of $160$ instead. Note that even though there is an outer automorphism of $\so (V \oplus V^*)$ that switches the two half-spinor representations, this outer automorphism does \emph{not} preserve the parabolic subalgebra, and thus the total ranks of the Betti numbers need not be preserved. In any case, adding up these two ranks tells us that the total rank of the Betti numbers of $\bigwedge^\bullet V$ in this case is $420$. 

On the other hand, we compute $\rho^{\rm top}_{(1/2,1/2)} = (3/2,1/2)$ and $\rho^{\rm bot}_{(1/2,1/2)} = (3/2,1/2)$.
The corresponding ``ranks'' are given by:
$$\dim L_{(3/2,1/2)}^{\fsp_2} = 35/4, \quad \dim L_{(3/2,1/2)}^{\Pin_2} = 12.$$
Notice that since $L_{(3/2,1/2)}^{\fsp_2}$ is only formally defined via the symplectic Weyl character formula, there is no reason that setting all $y_i = 1$ should give an integer (and indeed in this case it does not). Our formula for the total rank still applies, however, and gives $2^2 \cdot 35/4 \cdot 12 = 420$.
\end{example}

\begin{example}\label{ex:evenSpinorOddDim}
Let \( n \) be arbitrary, and let \( V \) be a vector space of dimension \( 2n + 1 \). Consider the algebra \( A = \Sym(\bigwedge^2 V) \), and the half-spinor representation
\[
L_{\omega_{4n+1}}^{\fso_{4n+2}} = \bigwedge^{\mathrm{even}} V.
\]
The minimal free resolution of this \( A \)-module has terms of the form
\[
\bS_\lambda V \otimes A(-|\lambda|/2),
\]
where \( \lambda \) is a self-conjugate partition of even rank. 

Explicitly, assume $n=2$. The highest weights of the representations (up to $1/2$ times a trace representation) appearing in the minimal free resolution are given by:
\[
\begin{aligned}
&\{0,\ 0,\ 0,\ 0,\ 0\}, && \{2,\ 2,\ 0,\ 0,\ 0\}, && \{3,\ 2,\ 1,\ 0,\ 0\}, && \{3,\ 3,\ 2,\ 0,\ 0\}, \\
&\{4,\ 2,\ 1,\ 1,\ 0\}, && \{4,\ 3,\ 2,\ 1,\ 0\}, && \{4,\ 4,\ 2,\ 2,\ 0\}, && \{4,\ 4,\ 4,\ 4,\ 0\}, \\
&\{5,\ 2,\ 1,\ 1,\ 1\}, && \{5,\ 3,\ 2,\ 1,\ 1\}, && \{5,\ 4,\ 2,\ 2,\ 1\}, && \{5,\ 4,\ 4,\ 4,\ 1\}, \\
&\{5,\ 5,\ 2,\ 2,\ 2\}, && \{5,\ 5,\ 4,\ 4,\ 2\}, && \{5,\ 5,\ 5,\ 4,\ 3\}, && \{5,\ 5,\ 5,\ 5,\ 4\}
\end{aligned}
\]
Taking ranks yields the following sequence:
\[
\begin{aligned}
\{\, 
1,\ 50,\ 280,\ 315,\ 450,\ 1024,\ 560,\ 70, \\
224,\ 700,\ 720,\ 160,\ 175,\ 126,\ 40,\ 5 
\,\}
\end{aligned}
\]
in which case the total rank may be computed as $4900$. On the other hand, we compute:
$$\rho^{\rm top}_{(1/2,1/2,1/2,1/2,1/2)} = (3/2,1/2), \quad \rho^{\rm bot}_{(1/2,1/2,1/2,1/2,1/2)} = (5/2,3/2,1/2).$$
The corresponding ranks are given by $\dim L_{(3/2,1/2)}^{\fsp_4} = 35/4$ and $\dim L_{(5/2,3/2,1/2)}^{\Pin_6} = 280$. The total rank of the even half-spinor representation is half of the total rank of the $\Pin_{10}$-representation of highest weight $(\frac12,\frac12,\frac12,\frac12,\frac12)$ (this is only true when $V$ is odd-dimensional). Our formula implies that the total rank equals
  $\frac{1}{2} \cdot 2^2 \cdot 35/4 \cdot 280 = 4900$.
\end{example}

\subsection{Exterior algebra}

For each $i$, let $h'(i) = (i-1,1^{i-1})$ be the hook partition. This is the transpose of $h(i-1)$ defined in the previous section (for $i\ge 2$); our convention is that $h'(1)$ is the empty partition.

\begin{lemma}\label{lem:typeDHookLemma}
  For a subset $S \subseteq [m]$, we have $-\beta_0(S)^{\rm op} = \bigcup_{i \in S} h'(m+1-i)$, where by union we mean to nest the corresponding hook shapes. 
\end{lemma}

We remark that it is irrelevant whether we specify that the subsets have even size or not since $h'(1)$ is empty; also, exactly one of $S \cup \{m\}$ and $S \setminus \{m\}$ has even size.

\begin{proof}
  Let \( S = \{s_1 < \cdots < s_k\} \subseteq [m] \), and let \( S^c = \{s_1' < \cdots < s_{m-k}'\} \) denote its complement in \( [m] \).
We have
  \begin{align*}
    -\beta_0(S)^{\rm op} &= (m-s_1, \dots, m+k-1-s_k, s'_{m-k}-(m-k), \dots, s_1' - 1).
  \end{align*}
  Now we proceed by induction: when $k = 0$ the result is clear.

  Now suppose $k \geq 1$. If $s_1=m$, then $-\beta_0(S)^{\rm op} = 0$ and there is nothing to show, so suppose otherwise. Then $s'_i=i$ for $i < s_1$, so that $-\beta_0(S)^{\rm op}$ has $m-s_1+1$ positive entries. Hence, if we delete the first entry and subtract 1 from the remaining $m-s_1$ positive entries, we are removing the hook $h'(m+1-s_1)$ from the partition $-\beta_0(S)^{\rm op}$. However, it is also clear that the resulting sequence is $-\beta_0(\{s_2,\dots,s_k\})^{\rm op}$ where $m$ is now $m-1$.
\end{proof}

We will now use Lemma~\ref{lem:typeDHookLemma} to convert subsets into unions of hook shapes: recall that this says that $-\beta_0(S)^{\rm op} = \bigcup h'(2n+1-i)$ where $h'(k)$ is the hook partition $(k-1,1^{k-1})$.

This exhibits one obvious symmetry: $-\beta_0(S)^{\rm op}$ and $-\beta_0(S^c)^{\rm op}$ are complementary partitions inside the $2n \times (2n-1)$ rectangle. In particular, the Schur functors are dual to one another up to a power of the determinant. We will restrict the action to the orthogonal Lie algebra, in which case the two Schur functors are isomorphic.

\begin{remark}
  When $\lambda=0$, we have $\rho^{\rm bot}_\lambda = \rho^{\rm top}_\lambda = (n-1,n-2,\dots,1)$. Furthermore, both have the same character as the orthogonal Schur functor $\bS_{[n-1,n-2,\dots,1]}(\bC^{2n})$ by Proposition~\ref{prop:stair-sosp}. This recovers Kostant's $\rho$-decomposition formula for $\fso_{2n}$.
  \end{remark}

Indeed, Theorem \ref{thm:kostantBCD} recovers \emph{both} the type B and D cases of Kostant's theorem:

\begin{corollary}[Type B/D Kostant $\rho$-decompositions]\label{cor:typeBKostant}
Let \(\fg = \mathfrak{so}_m\) and let $\rho = (\frac{m}{2}-1, \frac{m}{2}-2, \dots , \frac{m}{2}-n)$
denote the half-sum of the positive roots. Then there is a decomposition of $\so_{m}$-representations:
\[
  \bigwedge^\bullet (\so_{m}) \cong 2^n \cdot  (L^{\so_{m}}_\rho)^{\otimes 2}.
\]
\end{corollary}

\begin{proof}
    Note that the proof when $m$ is even is a direct consequence of Theorem \ref{thm:kostantBCD}, but the case when $m$ is odd is not as immediate. We consider $\det M'_\lambda(y_1,\dots,y_n;1)$ as in the proof of Proposition~\ref{prop:typeBCDdetForm}, but we will group terms differently this time. Recall that for general $\lambda$ this determinant can be written as a product:
\begin{align*}
  (-1)^{\binom{n+1}{2}} &
\det \begin{bmatrix}
y_1^{\lambda_2 + 2n - 1} - y_1^{-\lambda_2 - 2n + 1} & \cdots & y_1^{\lambda_{2n} + 1} - y_1^{-\lambda_{2n} - 1} \\
\vdots & \ddots & \vdots \\
y_n^{\lambda_2 + 2n - 1} - y_n^{-\lambda_2 - 2n + 1} & \cdots & y_n^{\lambda_{2n} + 1} - y_n^{-\lambda_{2n} - 1}
\end{bmatrix} \cdot \\
  &\qquad \det 
\begin{bmatrix}
y_1^{\lambda_1 + 2n} + y_1^{-\lambda_1 - 2n} & \cdots & y_1^{\lambda_{2n-1} + 2} + y_1^{-\lambda_{2n-1} - 2} & 1 \\
\vdots & \ddots & \vdots & \vdots \\
y_n^{\lambda_1 + 2n} + y_n^{-\lambda_1 - 2n} & \cdots & y_n^{\lambda_{2n-1} + 2} + y_n^{-\lambda_{2n-1} - 2} & 1 \\
2 & \cdots & 2 & 1
\end{bmatrix}.
\end{align*}
Set $\lambda = 0$. Let $M^{\rm bot}$ denote the first matrix that we are taking the determinant of. Subtracting $2$ times column $n+1$ from columns $1,\dots,n$ in the second matrix that appears, we find that this determinant can be written as $\det(M^{\rm top})$, where
\[
  M^{\rm top}=
\left[
y_i^{2n-2j+2}+y_i^{-2n+2j-2}-2
\right]_{1\le i,j\le n}.
\]
Next, by Lemma \ref{lem:adeltaIdentities} the denominator $a_{\delta}(y_1,\dots,y_n,y_1^{-1},\dots,y_n^{-1},1)$ can be written as
\[
  (-1)^{\binom{n+1}{2}}(b_{\rho^\rB}(y_1,\dots,y_n))^2\cdot \prod_{i=1}^n (y_i-y_i^{-1}).
\]
We claim that the quotient $\det(M^{\rm top})/\prod_{1\le i\le n}(y_i-y_i^{-1})$ is equal to $\det(M^{\rm bot})$; to see this, notice that by multilinearity of determinants there is an equality
\begin{align*}
  \frac{\det(M^{\rm top})}{\prod_{1\le i\le n}(y_i-y_i^{-1})}
  &= \det\left[
\frac{y_i^{2n-2j+2}+y_i^{-2n+2j-2}-2}{y_i-y_i^{-1}}
  \right]_{1\le i,j\le n}\\
  &= \det
\bigg[\sum_{\substack{1\le k\le 2n-2j+1\\ \text{odd}}}\left(y_i^k-y_i^{-k}\right)
\bigg]_{1\le i,j\le n},
\end{align*}
where the second equality follows from
\[
  \frac{y^{2d}+y^{-2d}-2}{y-y^{-1}}
  = \frac{(y^d-y^{-d})^2}{y-y^{-1}}
  = (y^d-y^{-d})(y^{d-1}+y^{d-3}+\cdots+y^{-d+1})
  = \sum_{\substack{1\le k\le 2d-1\\ \text{odd}}}(y^k-y^{-k}).
\]
Upon subtracting columns $j$ from column $j-1$ (in order) for $j=2,\dots,n$, we obtain the matrix $M^{\rm bot}$. It thus follows that the determinant we are interested in can be written as
\[
  \left(\frac{\det(M^{\rm bot})}{b_{\rho^\rB}(y_1,\dots,y_n)}\right)^2
  = (\ch L^{\fso_{2n+1}}_\rho)^2. \qedhere
\]
\end{proof}

\begin{example}
  For $n=2$, the terms in the proof of Theorem \ref{thm:kostantBCD} are regrouped as follows:
\[ \ytableausetup{boxsize = 0.5em}
\begin{tikzpicture}[>=latex]

  \node (top) at (0, 0) {\scriptsize 
  \begin{tabular}{rl}
      \textbf{$\varnothing$:} & \\
      $\varnothing \leadsto$ & $\varnothing = (0)$  \\
      \{2,4\} $\leadsto$ & {\color{red}\ydiagram{2,1,1}}$=2\,(1,1) + (2)$ \\
    \end{tabular}
  };

  \node (left) at (6, 0) {\scriptsize
    \begin{tabular}{rl}
      \textbf{\{1\}:} & \\
      \{\textbf{1},4\} $\leadsto$ & {\color{red}\ydiagram{3,1,1,1}} $=(0) + (2)$ \\
      \{\textbf{1},2\} $\leadsto$ & {\color{red}\ydiagram{3,3,2,2}} $=2 \,(1,1)$ \\
    \end{tabular}
  };

  \node (right) at (0, -2.5) {\scriptsize
    \begin{tabular}{rl}
      \textbf{\{3\}:} & \\
      \{\textbf{3},4\} $\leadsto$ & {\color{red}\ydiagram{1,1}} $=2 \,(1,1)$ \\
      \{2,\textbf{3}\} $\leadsto$ & {\color{red}\ydiagram{2,2,2}} $=(0) + (2)$ \\
    \end{tabular}
  };

  \node (bottom) at (6, -2.5) {\scriptsize
    \begin{tabular}{rl}
      \textbf{\{1,3\}:} & \\
      \{\textbf{1},\textbf{3}\} $\leadsto$ & {\color{red}\ydiagram{3,2,2,1}} $=2\,(1,1) + (2)$ \\
      \{\textbf{1},2,\textbf{3},4\} $\leadsto$ & {\color{red}\ydiagram{3,3,3,3}} $=(0)$ \\
    \end{tabular}
  };
\end{tikzpicture}  \qedhere
\]
\end{example}

\section{Type B Examples}\label{sec:typeBex}

\subsection{Dimension Formula}

\begin{proposition}\label{prop:explicit-U}
Assume the Type~B setup of \S\ref{sec:typeBCDrhodecomp}. Then there is an equality:
\begin{align*}
     \dim \rH_\bullet ( \fn_-; L_\lambda) &= 2^{\lceil m^2/2\rceil} 
    \!\!\!\!\!\prod_{\substack{1 \leq i < j \leq m\\ i \equiv j \pmod 2}} \left(1 + \frac{\lambda_{i} - \lambda_{j}}{j-i} \right) \cdot \Theta,
    \end{align*}
where $\Theta > 1$. 
\end{proposition}

\begin{proof}
      A similar method of proof can be used here, but we instead take advantage of the lower bound obtained by \cite{GKT} which says that (let $p=\lceil m^2/2 \rceil$):
      $$\dim \rH_\bullet (\fn_- ; \bC) \geq 2^{p} \cdot \begin{cases}
          2^{-1/2} m^{1/8} \kappa & m \ \text{odd},\\
          (m^2-1)^{1/16} \kappa & m \ \text{even},
      \end{cases}$$
      where $\kappa \approx 1.3814...$; notice that the terms $2^{-1/2} m^{1/8} \kappa$ and $(m^2-1)^{1/16} \kappa$ are both $> 1$ for all $m \geq 3$ odd and $m \geq 2$ even, respectively (the statement for $m=1$ is trivial so we avoid it). Thus we find that
      \begin{align*}
          \dim \rH_\bullet ( \fn_-; \bC) &= 2^p \cdot \frac{\Xi_1 \cdot \Xi_2}{2^{p - n(n+\eps) -1} }  \geq 2^p c
      \end{align*}
      where $c > 1$, implying that $\Xi_1 \cdot \Xi_2 > 2^{ p - n(n+\eps) -1} $. Since $\Xi_1 , \Xi_2$ can only increase as the entries of $\lambda$ are increased, we find that $\Xi_1 \cdot \Xi_2 > 2^{p - n(n+\eps) -1}$ for all choices of $\lambda$, and thus we take this as the desired expression for $\Theta$. 
\end{proof}

\begin{remark}
    As in Corollary \ref{cor:2n+2n-1 conj}, the above expression tells us that if $\lambda \neq 0$, then
    $$\dim \rH_\bullet ( \fn_-; L_\lambda) \geq \frac{3}{2} \cdot \dim \rH_\bullet ( \fn_-; \bC).$$
    This suggests more generally that the total dimension of the homology of any nontrivial representation $W$ of a finite dimensional nilpotent Lie algebra $\fn$ should satisfy
    \[
      \dim \rH_\bullet (\fn ; W) \geq \frac{3}{2} \cdot \dim \rH_\bullet (\fn ; \bC). \qedhere
    \]
\end{remark}

\begin{example}
Assume $\dim V = m = 2n = 4$, $\lambda = (2,1)$. The highest weights of the $\fgl(V)$-representations appearing in the minimal free resolution over \(A=\rU(\fn_-)=\rU(\bigwedge^2V^*\oplus V^*)\) are given by:
    \[
\begin{aligned}
&\{0,\ 0,\ -1,\ -2\}, && \{1,\ 0,\ -1,\ -2\}, && \{2,\ 1,\ -1,\ -2\}, && \{2,\ 2,\ -1,\ -2\}, \\
&\{4,\ 1,\ 1,\ -2\}, && \{4,\ 2,\ 1,\ -2\}, && \{4,\ 3,\ 2,\ -2\}, && \{4,\ 3,\ 3,\ -2\}, \\
&\{6,\ 1,\ 1,\ 0\}, && \{6,\ 2,\ 1,\ 0\}, && \{6,\ 3,\ 2,\ 0\}, && \{6,\ 3,\ 3,\ 0\}, \\
&\{6,\ 5,\ 2,\ 2\}, && \{6,\ 5,\ 3,\ 2\}, && \{6,\ 5,\ 4,\ 3\}, && \{6,\ 5,\ 4,\ 4\}.
\end{aligned}
\]
Taking the ranks of all of these yields the list:
\[
\begin{aligned}
\{\, 
20,\ 64,\ 175,\ 140,\ 300,\ 540,\ 420,\ 189, \\
189,\ 420,\ 540,\ 300,\ 140,\ 175,\ 64,\ 20 
\,\}
\end{aligned}
\]
Adding all of these up we obtain a total rank of $3696$. On the other hand, we compute:
$\rho^{\rm top}_{(2,1)} = (7/2,1/2)$ and $\rho^{\rm bot}_{(2,1)} = (5/2,1/2)$.
The corresponding ``ranks" are given by:
$$\dim L_{(7/2,1/2)}^{\fsp_{4}} = \frac{77}{2}, \quad \dim L_{(5/2,1/2)}^{\Pin_{4}} = 24.$$
Our formula implies that the total rank equals $2^2 \cdot \frac{77}{2} \cdot 24 = 3696$.
\end{example}

\begin{remark}
    Notice that the total rank of the homology of the spinor representation is combinatorially equivalent to the total rank of the homology of a free $2$-step nilpotent Lie algebra, which is covered in the next section.
\end{remark}

\begin{remark}
This is the first case where the nilpotent radical \(\fn=V\oplus\bigwedge^2V\) is nonabelian; hence \(\rU(\fn_-)\) is
not a polynomial ring, and the complexes we get are parabolic BGG resolutions over \(\rU(\fn_-)\).
\end{remark}

\subsection{GKT Formula}

The following is an immediate consequence of Kostant's theorem in type B applied to the special case $\lambda = (0)$:

\begin{proposition}
    Let $V$ be an $m$-dimensional space. Then there is an equality
    $$\rH_\bullet (\bigwedge^2 V^* \oplus V^*; \bC) = \bigoplus_{\substack{\lambda \subseteq m \times m, \\ \lambda = \lambda^T}} \bS_\lambda (V).$$
\end{proposition}

For each $i$, let $h''(i) = (i,1^{i-1})$ denote the hook partition. The proof of the following lemma is identical to the proof of Lemma \ref{lem:typeCHookLemma}, so we omit it and only give a statement:

\begin{lemma}\label{lem:typeBHookLemma}
  For a subset $S \subseteq [m]$, we have $-\beta_0(S)^{\rm op} = \bigcup_{i \in S} h''(m+1-i)$, where by union we mean to nest the corresponding hook shapes. Here, $\beta_0$ and the reverse–negate operation $-\!(\cdot )^{\rm op}$ are as in \S\ref{sec:typeBCDidentities}. 
\end{lemma}

\begin{example}
The following example shows how to group terms for the homology of the 2-step nilpotent Lie algebra when $m = 4$. 
    {\ytableausetup{boxsize = 0.5em}
\begin{center}
\begin{tikzpicture}[>=latex]

  \node (top) at (0, 5.4) {\scriptsize 
  \begin{tabular}{rl}
      $\varnothing$: & \\
      $\varnothing \leadsto $ & $\varnothing \quad = (0)$ \\
      \{2\} $\leadsto$ & {\color{red}\ydiagram{3,1,1}} $= (1) + 2 \, (2,1) + (3) $ \\
      \{4\} $\leadsto$ & {\color{red}\ydiagram{1}}  $\quad =(1)$ \\
      \{2,4\} $\leadsto$ & {\color{red}\ydiagram{3,2,1}} $= 2\,(1,1) + 2\,(2) + 2\,(2,2) + 2\,(3,1)$ 
    \end{tabular}
  };

  \node (left) at (-4, 1.6) {\scriptsize
    \begin{tabular}{rl}
      \textbf{\{1\}:} & \\
      \{\textbf{1}\} $\leadsto$ & {\color{red}\ydiagram{4,1,1,1}} $= (1) + (3)$  \\
      \{\textbf{1},2\} $\leadsto$ & {\color{red}\ydiagram{4,4,2,2}} $= (0) + (2) + 2\,(2,2)$ \\
      \{\textbf{1},4\} $\leadsto$ & {\color{red}\ydiagram{4,2,1,1}} $= 2\,(1,1) + (2) + 2\,(3,1)$\\
      \{\textbf{1},2,4\} $\leadsto$ & {\color{red}\ydiagram{4,4,3,2}} $= (1) + 2\,(2,1)$ 
    \end{tabular}
  };

  \node (right) at (4, 1.6) {\scriptsize
    \begin{tabular}{rl}
      \textbf{\{3\}:} & \\
      \{\textbf{3}\} $\leadsto$ & {\color{red}\ydiagram{2,1}} $= (1) + 2\,(2,1)$\\
      \{2,\textbf{3}\} $\leadsto$ & {\color{red}\ydiagram{3,3,2}} $= 2\,(1,1) + (2) + 2\,(3,1) $ \\
      \{\textbf{3},4\} $\leadsto$ & {\color{red}\ydiagram{2,2}} $= (0) + (2) + 2\,(2,2) $\\
      \{2,\textbf{3},4\} $\leadsto$ & {\color{red}\ydiagram{3,3,3}} $= (1) + (3)$
    \end{tabular}
  };

  \node (bottom) at (0, -2.0) {\scriptsize
    \begin{tabular}{rl}
      \textbf{\{1,3\}:} & \\
      \{\textbf{1},\textbf{3}\} $\leadsto$ & {\color{red}\ydiagram{4,3,2,1}} $= 2\,(1,1) + 2\,(2) + 2\,(2,2) + 2\,(3,1)$ \\
      \{\textbf{1},2,\textbf{3}\} $\leadsto$ & {\color{red}\ydiagram{4,4,4,3}} $=(1)$ \\
      \{\textbf{1},\textbf{3},4\} $\leadsto$ & {\color{red}\ydiagram{4,3,3,1}} $= (1) + 2 \, (2,1) + (3)$\\
      \{\textbf{1},2,\textbf{3},4\} $\leadsto$ & {\color{red}\ydiagram{4,4,4,4}} $=(0)$ 
    \end{tabular}
  };

  \draw[-] (top) -- (left);
  \draw[-] (top) -- (right);

  \draw[-] (left) -- (bottom);
  \draw[-] (right) -- (bottom);
\end{tikzpicture}
\end{center}} 
We can verify directly that each block adds up to the same representation.
\end{example}

The total rank of this homology was computed in \cite{GKT}, which we can recover:

\begin{lemma}\label{lem:GKTrankFormula}
    There is an equality:
    \[
\ch \rH_\bullet(\bigwedge^2 V^* \oplus V^*;\,\bC)=
\begin{cases}
2^n\ \big(\ch L^{\fsp_{2n}}_{(n-\frac12,\dots,\frac12)}\big)\,
       \big(\ch L^{\Pin_{2n}}_{(n-\frac12,\dots,\frac12)}\big), & m=2n,\\[4pt]
2^{n+1}\ \big(\ch L^{\fso_{2n+1}}_{(n,n-1,\dots,1)}\big)^{\!2}, & m=2n+1.
\end{cases}
\]
\end{lemma}

\begin{proof}
    The proof when $m = 2n$ is even is just a direct translation of Theorem \ref{thm:kostantBCD}. When $m = 2n+1$ is odd, the proof is essentially identical to the proof of Corollary \ref{cor:typeBKostant}.
  \end{proof}

\bibliographystyle{alpha}
\bibliography{biblio}
  
\end{document}